\documentclass[11pt,reqno]{amsart}
\usepackage{}
\usepackage[latin1]{inputenc}
\usepackage{amssymb}
\usepackage{mathrsfs}
\usepackage{amsfonts}
\usepackage{amsfonts,amssymb,amsmath,amsthm}
\usepackage{url}
\usepackage{enumerate}
\usepackage[pdftex,bookmarksnumbered]{hyperref}
\usepackage{graphicx,xcolor}
\newtheorem{theorem}{Theorem}[section]
\newtheorem{lemma}[theorem]{Lemma}
\newtheorem{proposition}[theorem]{Proposition}
\newtheorem{corollary}[theorem]{Corollary}
\theoremstyle{definition}
\newtheorem{definition}[theorem]{Definition}
\newtheorem{remark}{Remark}
\numberwithin{equation}{section}

\newtheorem{example}{Example}

\usepackage[left=2.0cm, right=2.0cm, top=2.0cm, bottom=2.0cm]{geometry}

\usepackage{graphicx}%
\usepackage{color}
\usepackage{cite}
\usepackage{hyperref}
\hypersetup{
	colorlinks = true,%
	citecolor = blue,
	filecolor=red,%
	linkcolor = [rgb]{0.65,0.0,0.0},%
	anchorcolor = red,
	pagecolor = red,
	urlcolor= [rgb]{0.65,0.0,0.0}
	linktocpage=true,
	pdfpagelabels=true,
	bookmarksnumbered=true,
}

\DeclareMathOperator{\inj}{inj}

\DeclareMathOperator{\diam}{diam}

\DeclareMathOperator{\dvol}{dvol}
\DeclareMathOperator{\vol}{vol}

\DeclareMathOperator{\loc}{loc}

\DeclareMathOperator{\di}{div}

\DeclareMathOperator{\lo}{loc}

\DeclareMathOperator{\He}{Hess}

\author{Canjun Meng}
\address{
Department of Mathematics\\
East China University of Science and Technology\\
200237 Shanghai, China}
\email{10173968@mail.ecust.edu.cn
}

\author{Han Wang}
\address{
Department of Mathematics\\
East China University of Science and Technology\\
200237 Shanghai, China}
\email{10172443@mail.ecust.edu.cn
}

\author{Wei Zhao}
\address{
Department of Mathematics\\
East China University of Science and Technology\\
200237 Shanghai, China}
\email{szhao\underline{ }wei@yahoo.com}

\begin{document}

\title[Hardy type inequalities on closed  manifolds via Ricci curvature]{Hardy type inequalities on closed  manifolds via Ricci curvature}

\begin{abstract}
The paper is devoted to Hardy type inequalities on closed manifolds. By means of various weighted Ricci curvatures, we establish several sharp Hardy type inequalities on closed weighted Riemannian manifolds.
Our results complement in several aspects those obtained recently in the noncompact   Riemannian
setting.
\end{abstract}
\maketitle

\section{Introduction} \label{sect1}
The classical Hardy inequality states that for any $p\in (1,n)$,
\begin{align*}
\int_{\mathbb{R}^n} |\nabla u|^pdx\geq \left(\frac{n-p}{p}\right)^p\int_{\mathbb{R}^n} \frac{|u|^p}{|x|^p}dx,\ \forall\,u\in C^\infty_0(\mathbb{R}^n),\tag{1.1}\label{1.1newone}
\end{align*}
where
  $\left[{(n-p)}/{p}\right]^p$ is sharp (see for instance   Hardy et al. \cite{HPL}). It is well-known that Hardy inequalities play a prominent role in the theory of linear and nonlinear partial differential equations. See e.g. \cite{BCC,BV,CM,DA,PV,V,VZ} and references therein.

 In recent years, a lot of  effort  has been
devoted to the study of Hardy inequalities in curved spaces.  As far as we know, Carron \cite{Ca} was the first who studied weighted $L^2$-Hardy inequalities  on Riemannian manifolds. Inspired by \cite{Ca}, a systematic study of the Hardy inequality is carried out by Berchio,  Ganguly and  Grillo \cite{BGD}, D'Ambrosio and Dipierro \cite{DD},   Kombe and \"Ozaydin \cite{KO,KO2}, Yang, Su and Kong \cite{YSK}, etc.
In the aforementioned works, a complete  non-compact  Riemannian manifold, which is usually of  nonpositive sectional curvature,  is a necessary condition.
 So this raises naturally a question as to how to establish the Hardy inequalities on a closed   manifold (i.e., a compact manifold without boundary) of nonnegative sectional/Ricci curvature?

 In view of Myers' compactness theorem, the assumption of the question above  is natural. However, up to now, no result is available
in the literature concerning this problem. The main issue is that the Hardy inequality  is generally  invalid for constant functions (e.g. (\ref{1.1newone})), but such functions are smooth functions with compact support on a closed manifold. Another issue is a technical difficulty which raises from the fact that the Laplacian of the distance function is bounded above
when the sectional/Ricci curvature is nonnegative. However, one usually needs the lower bound  of the Laplacian to establish
the Hardy inequality (cf. \cite{DD,KR,KS,YSK}, etc).

On the other hand,  although there has been tremendous interest in developing the Hardy inequalities in the Riemannian framework,  limited work has been done in the category of weighted Riemannian manifolds.
Recall that a triple $(M,g,d\mu)$ is called  a  weighted Riemannian manifold if $(M,g)$ is a Riemannian manifold endowed with a smooth measure $d\mu$. For convenience, express $d\mu$ by $e^{-\Psi}\dvol_g$, where $\Psi$ is a smooth function on $M$ and $\dvol_g$ is the standard Riemannian measure induce by $g$.
According to Lott and Villani \cite[Definition 7.1]{LV} (also see \cite{ST,WW22}),  the weighted Ricci curvature, {$\mathbf{Ric}_N$}, is defined as follows:
\begin{align*}\mathbf{Ric}_N=\left\{
\begin{array}{lll}
\mathbf{Ric}+\He (\Psi), && \text{ for }N=\infty,\\
\\
\mathbf{Ric}+\He(\Psi)-\frac{1}{N-n}d\Psi\otimes d\Psi, && \text{ for }n<N<\infty,\tag{1.2}\label{RiicNdonete}\\
\\
\mathbf{Ric}+\He(\Psi)-\infty \left(d\Psi\otimes d\Psi\right), && \text{ for } N=n.
\end{array}
\right.
\end{align*}
Hereafter, $\mathbf{Ric}$ denotes  the standard Ricci curvature of $(M,g)$.
The weighted Ricci curvature has an important influence on the geometry of manifolds. For example, Myers' compactness theorem remains valid for weighted Riemannian manifolds with positive weighted Ricci curvature (cf. \cite{Ba,L,LV,O,WW22}, etc).
But as far as we know, there is  no results concerning the relation between the weighted Ricci curvature and the Hardy inequality.

The purpose of the present paper is  to address the questions above. In particular, we
 establish several sharp Hardy type inequalities on  closed weighted Riemannian manifolds with nonnegative weighted Ricci curvature.

In order to state our main results, we introduce a notation first.
Given a point $o$ in a closed  manifold $M$, set
\[
C^\infty(M,o):=\{u\in C^\infty(M):\, u(o)=0\}.
\]
Our first result reads as follows.

\begin{theorem}\label{reverRicinfty}
Let $(M,g,d\mu)$ be an $n(\geq 2)$-dimensional closed  weighted Riemannian manifold with $\mathbf{Ric}_N\geq 0$ for some $N\in [n,\infty)$. Given a point $o\in M$, set $r(x):=\text{dist}(o,x)$, i.e., the distance between $o$ and $x$.
For any $p,\beta\in \mathbb{R}$ with $p\in (1,N)\cup(N,\infty)$, $\beta<-N$ and $p+\beta>-n$,
we have
 \[
\int_{M}|\nabla u |^pr^{p+\beta} {d\mu}\geq \left( \frac{|N+\beta|}{p} \right)^p\int_{M} {|u|^p}r^{\beta}{d\mu}, \ \forall\, u\in C^\infty(M,o).\tag{1.3}\label{RicNr}
\]
Moreover,  if $N=n$, then $\left( \frac{|N+\beta|}{p} \right)^p$ is sharp with respect to $C^\infty(M,o)$, i.e.,
\[
\left( \frac{|n+\beta|}{p} \right)^p=\inf_{u\in  C^\infty(M,o)\backslash\{0\}}\frac{\int_{{M}}|\nabla u|^p r^{p+\beta}d \mu}{\int_M |u|^p r^{\beta} {d\mu}}.
\]
\end{theorem}
In the case when $\Psi=\text{const}$, (\ref{RicNr}) is not optimal if $N\neq n$.
 On the other hand, if $N=n$,  bounding
$\mathbf{Ric}_N$ from below makes sense only if $\Psi=\text{const}$, in which case $\mathbf{Ric}_n=\mathbf{Ric}$. Hence,  (\ref{RicNr}) implies a sharp Hardy type inequalities on
 a closed Riemannian manifold $(M,g)$ with nonnegative Ricci curvature. That is,
for any $p>n$,
\begin{align*}
\int_{M}|\nabla u|^p \dvol_g&\geq  \left(\frac{p-n}p \right)^p\int_{M}  \frac{|u|^p}{r^{p}} \dvol_g,\ \ \forall\, u\in C^\infty(M,o).\tag{1.4}\label{newhary1.1}
\end{align*}
By comparing (\ref{newhary1.1}) with (\ref{1.1newone}), we see that $C^\infty(M,o)$ is  a  good candidate for establishing the Hardy type inequality on closed manifolds.

Now we consider the case of $\infty$-Ricci curvature. Recall $d\mu=e^{-\Psi}\dvol_g$.
Since $(M,g)$ is closed, both the upper bound of $|\Psi|$ and the lower bound of $\partial_r\Psi:=\frac{\partial \Psi(r,y)}{\partial r}$ always exist, where  $(r,y)$ denotes the polar coordinate system around $o$.
Then we have the following result.
\begin{theorem}\label{secondrpesenthe}
Let $(M,g,d\mu)$ be an $n(\geq 2)$-dimensional closed  weighted Riemannian manifold with $\mathbf{Ric}_\infty\geq 0$. Given a point $o\in M$, let $r(x):=\text{dist}(o,x)$. Set $d\mu:=e^{-\Psi}\dvol_g$.

\smallskip

(i) Suppose $\partial_r\Psi\geq -\lambda$ for some $\lambda\geq 0$.
Thus, for any $p,\beta\in \mathbb{R}$ and $\alpha\in \mathbb{R}\backslash\{0\}$ with
\[
p>\max\{1,-n-\beta\},\ \beta+1<(\alpha-1)(p-1)\leq -[n-1+\lambda \diam(M)],
\]
we have
 \[
\int_{M} | \nabla u|^p r^{p+\beta} d\mu\geq (\vartheta_{\alpha,\beta,p})^p\int_{M}{|u|^p}r^{\beta} d\mu,\ \forall\, u\in C^\infty(M,o),
\]
where $\vartheta_{\alpha,\beta,p}:={[(\alpha-1)(p-1)-\beta-1]}/{p}$.
In particular, $(\vartheta_{\alpha,\beta,p})^p$ is sharp if $\lambda=0$ and $\alpha=\frac{p-n}{p-1}\neq0$.

\smallskip

(ii)
Suppose $|\Psi|\leq k$ for some $k\geq 0$. Thus, for any
$p,\beta\in \mathbb{R}$ and $\alpha\in \mathbb{R}\backslash\{0\}$ with
\[
p>\max\{1,-n-\beta\},\ \beta+1<(\alpha-1)(p-1)\leq -\left(n-1+4k\right),
\]
 we have
 \[
\int_{M} | \nabla u|^pr^{p+\beta} d\mu\geq (\vartheta_{\alpha,\beta,p})^p\int_{M} {|u|^p}r^{\beta}d\mu,\ \forall\, u\in C^\infty(M,o),
\]
where $\vartheta_{\alpha,\beta,p}$ is defined as above.
In particular, $(\vartheta_{\alpha,\beta,p})^p$ is sharp if $k=0$ and
$\alpha=\frac{p-n}{p-1}\neq0$.
\end{theorem}
We also obtain logarithmic versions of Theorems \ref{reverRicinfty} and \ref{secondrpesenthe}. See Theorems \ref{logriccnonclosed2} and \ref{logriccnonRicinf2closed} below. On the other hand,
for any $u\in C^\infty(M)$, obviously $u-u(o)\in C^\infty(M,o)$. Thus,
we have
another version of (\ref{newhary1.1}).
\begin{theorem}\label{newharyineqtype}
Let $(M,g)$ be an $n(\geq 2)$-dimensional closed  Riemannian manifold  with nonnegative Ricci curvature. Given a point $o\in M$, set $r(x):=\text{dist}(o,x)$. For any $p>n$, we have
\begin{align*}
\int_{M}|\nabla u|^p \dvol_g&\geq  \left(\frac{p-n}p \right)^p\int_{M}  \frac{|u-u(o)|^p}{r^{p}} \dvol_g,\ \ \forall\, u\in C_0^\infty(M)=C^\infty(M).
\end{align*}
In particular, $\left(\frac{p-n}p \right)^p$ is sharp (but not achieved) in the following sense
\[
\left(\frac{p-n}p \right)^p=\inf_{u\in  C^\infty(M)\backslash \mathcal {I}_o}\frac{\int_{{M}} |\nabla u|^p \dvol_g}{\int_M \frac{|u-u(o)|^p}{r^p} {\dvol_g}},
\]
where $\mathcal {I}_o:=\{u\in C^\infty(M):\, u\equiv u(o)\}$.
\end{theorem}

The paper is organized as follows. We devote Section \ref{sechardpunct}  to Hardy inequalities on  punctured weighted Riemannian manifolds.
Hardy type inequalities on closed weighted Riemannian manifolds are investigated in Section \ref{distanceHardy}, where
the proofs of Theorem
\ref{reverRicinfty}-\ref{newharyineqtype} are given. We present an example in Section \ref{exapmlesec}.
And Appendix \ref{app} is devoted to several lemmas which are used throughout the previous sections.

\section{Hardy inequalities on punctured manifolds}\label{sechardpunct}

In the sequel, the quadruple $(M,o,g,d\mu)$ is called an  {\it $n(\geq2)$-dimensional closed pointed weighted Riemannian manifold} if $(M,g)$ is an $n$-dimensional closed Riemannian manifold equipped with a smooth measure $d\mu$ and $o$ is a point in $M$.
For such a space, we always set $r(x):=\text{dist}(o,x)$ and $M_o:=M\backslash\{o\}$.

Express $d\mu=e^{-\Psi}\dvol_g$, where $\Psi\in C^\infty(M)$ and $\dvol_g$ is the standard Riemannian measure induced by $g$.
Let $\Delta_\Psi:=\Delta-\nabla\Psi\cdot\nabla$ denote the Bakry-\'Emery Laplacian operator, where $\Delta$ is the standard Laplacian operator.

Let $\mathbf{Ric}_N$, $N\in [n,\infty]$ be the weighted Ricci curvature defined as in (\ref{RiicNdonete}). Given $\lambda\geq 0$, we use the notation $\partial_r\Psi\geq -\lambda$ to denote ${\partial \Psi(r,y)}/{\partial r}\geq -\lambda$ along all the minimal geodesics from $o$, where $(r,y)$ is the polar coordinate system around $o$. Then
one has the following comparison theorem. See \cite{Ba,LV,O,WW22}, etc. for the proofs.

\begin{lemma}\label{welaplaccompar}Let $(M,o,g, d\mu)$ be an $n$-dimensional pointed weighted closed Riemannian manifold.

\smallskip

(1) If $\mathbf{Ric}_N\geq 0$ for some $N\in [n,\infty)$, then $\Delta_\Psi r\leq \frac{N-1}r$ a.e. on $M$.

\smallskip

(2) If $\mathbf{Ric}_\infty\geq 0$ and $\partial_r\Psi\geq -\lambda$ for some $\lambda\geq 0$, then $ \Delta_\Psi r\leq \frac{n-1}r+\lambda$ a.e. on $M$.

\smallskip

(3) If $\mathbf{Ric}_\infty\geq 0$ and $|\Psi|\leq k$ for some $k\geq 0$, then $ \Delta_\Psi r\leq \frac{n-1+4k}r$ a.e. on $M$.

\end{lemma}

Now we explain $\Delta_\Psi r$ from a geometric point of view.
Let $S_oM$ be the unit tangent sphere in $T_oM$, i.e., $S_oM:=\{y\in T_oM: \,|y|=1\}$. For any $y\in S_oM$,  denote by $i_y$ the cut value of $y$.  Let $(r,y)$ be the polar coordinate system around $o$. Write
\begin{align*}
d\mu:=\sigma_o(r,y)\,dr\, d\nu_o(y),\ 0<r<i_y,\ y\in S_oM,
\end{align*}
where $d\nu_o(y)$ is the Lebesgue measure on $S_oM$, i.e., $\int_{S_oM}d\nu_o(y)=\vol(\mathbb{S}^{n-1})=:c_{n-1}$. Then one has
\[
\Delta_\Psi r=\frac{\partial}{\partial r}\log\sigma_o(r,y).\tag{2.1}\label{laplacianvolume}
\]
Thus, Lemma \ref{welaplaccompar} yields the following result.
\begin{lemma}\label{volumecompacrions}
Let $(M,o,g, d\mu)$ be an $n$-dimensional pointed weighted closed Riemannian manifold. Denote by $(r,y)$ the polar coordinate system around $o$.

\smallskip

(1) If $\mathbf{Ric}_N\geq 0$ for some $N\in [n,\infty)$, then for each $y\in S_oM$, the function $f_{1,y}(r):=\frac{\sigma_o(r,y)}{r^{N-1}}$ is  non-increasing for $0<r<i_y$. Moreover, $\lim_{r\rightarrow0^+}f_{1,y}(r)$ exists if and only if $N=n$.

(2) If $\mathbf{Ric}_\infty\geq 0$ and $\partial_r\Psi\geq -\lambda$ for some $\lambda\geq 0$, then for each $y\in S_oM$, the function $f_{2,y}(r):=\frac{\sigma_o(r,y)}{e^{\lambda r}r^{n-1}}$ is  non-increasing for $0<r<i_y$. In particular, we have
\[
d\mu\leq e^{-\Psi(o)+\lambda r}r^{n-1}drd\nu_o(y), \ \forall\, 0<r<i_y,\ y\in S_oM.
\]

(3) If $\mathbf{Ric}_\infty\geq 0$ and $|\Psi|\leq k$ for some $k\geq 0$, then for each $y\in S_oM$, the function $f_{3,y}(r):=\frac{\sigma_o(r,y)}{r^{n+4k-1}}$ is  non-increasing for $0<r<i_y$. Moreover, $\lim_{r\rightarrow0^+}f_{3,y}(r)$ exists if and only if $k=0$.

\end{lemma}
\begin{proof} The proof is similar to the one of the standard volume comparison theorem. Hence, we only present a sketch of the proof.

(1)  Lemma \ref{welaplaccompar} (1) together with (\ref{laplacianvolume}) yields
\[
\frac{\partial}{\partial r}\log \left( \frac{ \sigma_o(r,y)}{  r^{N-1}} \right)\leq 0\Longrightarrow \frac{d}{dr}f_{1,y}(r)\leq 0,
\]
i.e., $f_{1,y}(r)$ is non-increasing. Express $\dvol_g:= g(r,y)\, dr d\nu_o(y)$. Since $g(r,y)\sim r^{n-1}$ (cf. Chavel \cite{IC}) and $N\geq n$, it is not hard to see that   $\lim_{r\rightarrow0^+}f_{1,y}(r)$ exists if and only if $N=n$.

(2) Lemma \ref{welaplaccompar} (2) together with (\ref{laplacianvolume}) yields
\[
\frac{\partial}{\partial r}\log \left( \frac{ \sigma_o(r,y)}{ e^{\lambda r} r^{n-1}} \right)\leq 0\Longrightarrow \frac{d}{dr}f_{2,y}(r)\leq 0.
\]
The second statement follows from $\lim_{r\rightarrow 0^+}f_{2,y}(r)=e^{-\Psi(o)}$ directly.

(3) The proof is almost the same as (1) and hence, we omit it.
\end{proof}

\begin{remark}In fact,
Lemmas \ref{welaplaccompar}-\ref{volumecompacrions} are  valid for all complete weighted Riemannian manifolds.
\end{remark}

Now we recall the $p$-Laplacian operator.
Given $p>1$, the {\it  weighted $p$-Laplacian}  with respect to $d\mu$ is defined as
\[
\Delta_{\mu,p}(f):=\di_\mu\left(  |\nabla f|^{p-2} \nabla f  \right),
\]
where $\di_\mu$, the {\it divergence operator with respect to $d\mu$},  is defined as
\[
{\di}_\mu(X)d\mu:=d(X\rfloor d\mu),\ \forall \,X\in C^1(TM).\tag{2.2}\label{divegedfe}
\]
In particular, $\Delta_{\mu,p}(f)= e^{\Psi}\di_{\vol_g}\left( e^{-\Psi} |\nabla f|^{p-2} \nabla f  \right)$ and $\Delta_{\mu,2}(f)=\Delta_\Psi f$.

Following D'Ambrosio et al. \cite{DD},  we say that a function $\rho(x)\in W^{1,p}_{\lo}({M_o})$ satisfies {\it $-\Delta_{\mu,p}(c\rho)\geq 0$ in the weak sense} for some $c\in \mathbb{R}$ if
\[
c\int_{M_o} |\nabla \rho|^{p-2}g(\nabla \rho,\nabla u) d\mu\geq 0, \ \forall \,u\in C^1_0({M_o}) \text{ with }u\geq 0.\tag{2.3}\label{weak2}
\]
A vector filed $X$ is said to {\it belong to} $L^1_{\lo}(TM_o)$ if  $\int_K |X|d\mu$ is finite for any compact set $K\subset M_o$.
Given a vector filed $X\in L^1_{\lo}(T{M_o})$ and a nonnegative function  $f_X\in L^1_{\lo}({M_o})$, we say  {\it $f_X\leq \di_\mu X$ in the weak sense} if
\[
\int_{M_o} u f_X d\mu\leq -\int_{M_o}   g(\nabla u ,X)  d\mu, \ \forall \,u\in C^1_0({M_o}) \text{ with }u\geq 0.\tag{2.4}\label{weak1}
\]

\begin{remark}Owing to (\ref{divegedfe}), both (\ref{weak2}) and (\ref{weak1}) are well-defined. However,
neither of them makes sense for $u\in C_0^1(M)=C^1(M)$ since $M$ has no boundary. For example, consider a positive constant function  on $M$.
\end{remark}

Note that (\ref{divegedfe}) implies
\[
\int_{M_o} u \di_{\mu}X d\mu=-\int_{M_o} g(\nabla u,X)d\mu\leq \int_{M_o} |\nabla u||X| d\mu,\ \forall\, u\in C^1_0(M_o),
\]
which together with the proof of D'Ambrosio et al. \cite[Lemma 2.10]{DD}  furnishes the following result directly.

\begin{lemma}\label{divlemf} Let $(M,o,g,d\mu)$ be a closed pointed weighted Riemannian manifold.
Let $X\in L^1_{\lo}(T{M_o})$ be a vector filed  and let $f_X\in L^1_{\lo}({M_o})$ be a nonnegative function.
Given $p>1$, suppose the following conditions hold:

\smallskip

\ \ \ \ \ (i) $f_X\leq \di_{\mu} X$ in the weak sense;
\ \ \ \ \ (ii) $|X|^p/f_X^{p-1}\in L^1_{\lo}({M_o})$.

\smallskip

\noindent Then  we have
\begin{align*}
p^p {\int_{M_o}|\nabla u|^p\frac{|X|^p}{f_X^{p-1}} d\mu}    \geq{\int_{M_o} |u|^p f_Xd\mu},\ \forall\,u\in C^1_0({M_o}).
\end{align*}
\end{lemma}

We now generalize D'Ambrosio \cite[Theorem 2.7]{D} to  closed weighted Riemannian manifolds.

\begin{lemma}\label{mainlemmforcr}
Let $(M,o,g,d\mu)$ be a closed pointed weighted Riemannian manifold and let $\rho\in C^1(M_o)$ be a nonnegative nonconstant function.
 Suppose that $p\in(1,\infty)$, $\alpha,\beta\in \mathbb{R}$ and $\rho$  satisfy the following conditions:

\smallskip

 (1) $\rho^{(\alpha-1)(p-1)}|\nabla\rho|^{p-1},\rho^\beta|\nabla \rho|^p,\rho^{p+\beta}\in L^1_{\lo}({M_o})$;

 (2) $-\Delta_{\mu,p}(c\rho^\alpha)\geq 0$, where $c=\alpha[(\alpha-1)(p-1)-\beta-1]\neq 0$.

\smallskip

\noindent Then we have
\begin{align*}
\int_{{M_o}} |\nabla u|^p\rho^{p+\beta}d\mu\geq (\vartheta_{\alpha,\beta,p})^p\int_{{M_o}} |u|^p\rho^\beta |\nabla\rho|^p d\mu,\ \forall\,u\in C^1_0({M_o}),\tag{2.5}\label{5.1.1}
\end{align*}
 where $\vartheta_{\alpha,\beta,p}:={|(\alpha-1)(p-1)-\beta-1|}/{p}$.
\end{lemma}

\begin{proof}
Provided that $-\Delta_{\mu,p} (\rho^\alpha)\geq 0$ and $c>0$, we set
\[
X:=-\alpha \rho^{\beta+1}|\nabla\rho|^{p-2}\nabla \rho,\ \ \ f_X:=c |\nabla\rho|^p \rho^\beta.
\]
Clearly, $f_X\in L^1_{\lo}({M_o})$. The H\"older inequality together with Condition (1) implies $\rho^{\beta+1}|\nabla\rho|^{p-1}\in L^1_{\lo}({M_o})$ and hence,
$X\in L^1_{\lo}(T{M_o})$.
Moreover,  $|X|^p/f^{p-1}_X\in L^1_{\lo}({M_o})$ because $\rho^{p+\beta}\in L^1_{\lo}({M_o})$.

 Given $\epsilon>0$ and $u\in C^1_0({M_o})$ with $u\geq 0$, set $\rho_\epsilon:=\epsilon+\rho$ and $ v:=\rho_\epsilon^{-{c}/\alpha}u$.
 Condition (2) then furnishes $\int_{{M_o}}   |\nabla \rho^\alpha|^{p-2}g(\nabla \rho^\alpha, \nabla v) d\mu\geq 0$, that is,
\begin{align*}
c \int_{{M_o}}  \rho^{(\alpha-1)(p-1)}  \rho_\epsilon^{-\frac{c}\alpha-1}|\nabla\rho|^pu d\mu\leq \alpha\int_{{M_o}} \rho^{(\alpha-1)(p-1)}\rho_\epsilon^{-\frac{c}\alpha}|\nabla\rho|^{p-2}g( \nabla \rho,\nabla u) d\mu.\tag{2.6}\label{5.3newneed}
\end{align*}
Notice that
\begin{align*}
&\rho^{(\alpha-1)(p-1)}\rho_\epsilon^{-\frac{c}\alpha-1}|\nabla\rho|^p\leq \rho^\beta|\nabla\rho|^p\in L^1_{\lo}({M_o}),\\
&\rho^{(\alpha-1)(p-1)}\rho_\epsilon^{-\frac{c}\alpha}|\nabla\rho|^{p-2}|g(  \nabla \rho,\nabla u)|\leq \rho^{\beta+1}|\nabla\rho|^{p-1}|\nabla u|\in L^1_{\lo}({M_o}).
\end{align*}
The  Lebesgue  dominated convergence theorem together with (\ref{5.3newneed}) yields
\[
\int_{{M_o}} c\rho^\beta|\nabla\rho|^pud\mu\leq\int_{{M_o}}\alpha \rho^{\beta+1}|\nabla\rho|^{p-2}g( \nabla \rho,\nabla u)d\mu,
\]
that is, $f_X\leq \di_{\mu} X$ in the weak sense. Now (\ref{5.1.1}) follows from Lemma \ref{divlemf}.

If $-\Delta_{\mu,p} (\rho^\alpha)\leq 0$ and $c<0$, we set $X:=\alpha \rho^{\beta+1}|\nabla\rho|^{p-2}\nabla \rho,\ f_X:=-c |\nabla\rho|^p \rho^\beta$ and $v:=\rho_\epsilon^{-c/\alpha}u$. Then (\ref{5.1.1}) follows from
a similar argument.
\end{proof}

In order to study the sharpness of (\ref{5.1.1}), we
  introduce the following space.
\begin{definition}\label{DefDS}
Let $(M,o,g,d\mu)$, $p,\alpha,\beta$ and $\rho$  be as in Lemma \ref{mainlemmforcr}. Additionally suppose

\smallskip

(1) $\rho\neq0$ a.e. and $\rho^{q(p+\beta)}\in L^1_{\loc}(M_o)$ for some $q>1$;

(2) there exists a positive constant $C$ such that $\rho/|\nabla\rho|\leq C$.

\smallskip

\noindent Denote by $D^{1,p}(M_o,\rho^{p+\beta})$  the closure of  $C^\infty_0(M_o)$ with respect to the norm
\[
\|u\|_D:=\left( \int_{M_o} |\nabla u|^p \rho^{p+\beta} d\mu\right)^{\frac1p}.
\]
\end{definition}

Inspired by  D'Ambrosio \cite[Remark 2.8]{D},
we have the following lemma.

\begin{lemma}\label{secondlemmaforsharp}
Let $(M,o,g,d\mu)$, $p,\alpha,\beta$ and $\rho$  be as in Definition \ref{DefDS}. Additionally suppose

\smallskip

 (i) there exists some $s>0$ such that ${M_o}^-_s:=\rho^{-1}(-\infty,s]$ and ${M_o}^+_s:=\rho^{-1}(s,\infty)$ are non-empty subsets of $M_o$ with piecewise regular boundaries;

 (ii) there exists some $\epsilon_0>0$ such that
 \[
 0<\int_{{M_o}^-_s} \rho^{c(\epsilon)p+\beta} |\nabla \rho|^p d\mu<\infty,\ 0<\int_{{M_o}^+_s} \rho^{-c(\epsilon)p+\beta} |\nabla \rho|^p d\mu<\infty,\ \forall\,\epsilon\in (0,\epsilon_0),
 \]
where $c(\epsilon):= \frac{|(\alpha-1)(p-1)-\beta-1|+\epsilon}{p}$.

\smallskip

Set \begin{align*}v(x):=\left\{
\begin{array}{lll}
\left(\frac{\rho(x)}{s}\right)^{c(\epsilon)},&& \text{ if  }x\in{M_o}^-_s,\\
\\
\left(\frac{\rho(x)}{s}\right)^{-c(\epsilon/2)},&& \text{ if  }x\in {M_o}^+_s.
\end{array}
\right.
\end{align*}
If there exists a small $\delta_0>0$ such that $v_\delta:=\max\{v-\delta,0\}$ is  a globally Lipschitz function with compact support in $M_o$ for any $\delta\in (0,\delta_0)$, then  (\ref{5.1.1}) is sharp.

\end{lemma}
\begin{proof}
Since $c(\epsilon)>c(\epsilon/2)$, the assumptions (i) and (ii) yield
\begin{align*}
\infty>&c(\epsilon)^p\int_{M_o} |v|^p \rho^\beta |\nabla\rho|^pd\mu\\
=&c(\epsilon)^p\int_{{M_o}^-_s} \rho^{p+\beta}\frac{\rho^{(c(\epsilon)-1)p}}{s^{c(\epsilon)p}}|\nabla \rho|^pd\mu+c(\epsilon)^p\int_{{M_o}^+_s} \rho^{p+\beta}\frac{\rho^{(-c(\epsilon/2)-1)p}}{s^{-c(\epsilon/2)p}}|\nabla \rho|^pd\mu\\
=&\int_{M_o}\rho^{p+\beta}|\nabla v|^pd\mu+\left[ \left( \frac{c(\epsilon)}{c(\epsilon/2)} \right)^p-1 \right]\int_{{M_o}^+_s} \rho^{p+\beta}|\nabla v|^pd\mu,
\end{align*}
which implies $\|v\|_D<\infty$ and
\[
c(\epsilon)^p\int_{M_o} |v|^p \rho^\beta |\nabla\rho|^pd\mu>\int_{M_o} |\nabla v|^p\rho^{p+\beta}d\mu.\tag{2.7}\label{inequalforhard}
\]

If $v_\delta$ is  a globally Lipschitz function with compact support in $M_o$, Lemma \ref{lpschcom}  yields $v_\delta\in D^p(M_o,\rho^{p+\beta})$. Moreover,
a direct calculation together with $\|v\|_D<\infty$ furnishes
\begin{align*}
\|v_\delta-v\|_D^p
=\int_{\rho^{-1}\left(-\infty, s\delta^{\frac{1}{c(\epsilon)}}\right]}+\int_{ \rho^{-1}\left[s\delta^{-\frac{1}{c(\epsilon/2)}},+\infty\right)}|\nabla v|^p\rho^{p+\beta} d\mu\rightarrow0, \text{ as }\delta\rightarrow0^+,
\end{align*}
which implies $v\in D^p(M_o,\rho^{p+\beta})$. This fact together with Lemma \ref{controlnormandspace} yields a sequence $v_j\in C^\infty_0(M_o)$ such that
\[
\int_{M_o} |v_j|^p \rho^\beta |\nabla\rho|^pd\mu\rightarrow \int_{M_o} |v|^p \rho^\beta |\nabla\rho|^pd\mu,\ \ \ \int_{M_o} |\nabla v_j|^p\rho^{p+\beta}d\mu\rightarrow \int_{M_o} |\nabla v|^p\rho^{p+\beta}d\mu.
\]
Due to (\ref{5.1.1}) and (\ref{inequalforhard}), for any $\epsilon>0$, there is a large $j$ such that
\begin{align*}
(\vartheta_{\alpha,\beta,p})^p\leq \inf_{u\in C^1_0({M_o})\backslash\{0\}}\frac{\int_{{M_o}}|\nabla u|^p \rho^{p+\beta}d\mu}{\int_{{M_o}}|u|^p \rho^\beta |\nabla\rho|^pd\mu} \leq \lim_{j\rightarrow\infty}\frac{\int_{{M_o}} |\nabla v_j|^p\rho^{p+\beta}d\mu}{\int_{{M_o}}|v_j|^p \rho^\beta |\nabla\rho|^pd\mu}< c(\epsilon)^p\rightarrow (\vartheta_{\alpha,\beta,p})^p,
\end{align*}
as $\epsilon\rightarrow 0^+$.
So the sharpness of (\ref{5.1.1}) follows.
\end{proof}

 The  following result is   useful in this paper, although the proof is trivial.
\begin{lemma}\label{interglemaGa}Given $d>s>0$,
for any $k_1,k_2\in \mathbb{R}$ and $l\in(s,d\,]$, set
\[
H_1(k_1,k_2):=\int^s_0 \left[ \log\left( \frac{d}{r} \right)  \right]^{k_1} r^{k_2} dr,\ H_2(l,k_1,k_2):=\int^l_s \left[ \log\left( \frac{d}{r} \right)  \right]^{k_1} r^{k_2} dr.
\]
Thus,
$H_1$ is well-defined if
\begin{align*}\left\{
\begin{array}{lll}
k_1\in \mathbb{R}&& \text{ if  }k_2>-1,\\
\tag{2.8}\label{log1use}\\
k_1<-1&& \text{ if } k_2 =-1,
\end{array}
\right.
\end{align*}
while $H_2$ is well-defined if
\begin{align*}\left\{
\begin{array}{lll}
k_1,k_2\in \mathbb{R}&& \text{ if  }l<d,\\
\tag{2.9}\label{log2use}\\
k_1>-1,k_2\in \mathbb{R}&& \text{ if  }l=d.
\end{array}
\right.
\end{align*}
\end{lemma}

In the sequel, $\inj_o$ denotes the injectivity radius of $o$, while $B_s(o)$ denotes an open geodesic ball of radius $s$ centered at $o$.

 \begin{theorem}\label{logriccnonclosed}
 Let $(M,o,g,d\mu)$ be an $n$-dimensional closed pointed weighted Riemannian manifold with $\mathbf{Ric}_N\geq 0$ for some $N\in [n,\infty)$.
Suppose that  $d,p,\beta\in \mathbb{R}$ and $\alpha\in \mathbb{R}\backslash\{0\}$ satisfy
\[
d\geq \diam(M), \ p\geq N, \  \log\left( \frac{d}{\diam(M)} \right)(N-p)\leq (\alpha-1)(p-1)< \beta+1.\tag{2.10}\label{logconditnum}
\]
Then we have
\begin{align*}
\int_{{M_o}} \left[\log\left(\frac{d}{r} \right)\right]^{p+\beta}|\nabla u|^p  {d\mu}\geq (\vartheta_{\alpha,\beta,p})^p\int_{{M_o}} \left[\log\left(\frac{d}{r} \right)\right]^{\beta} \frac{|u|^p}{r^p} {d\mu},\ \forall\,u\in C^\infty_0({M_o}),\tag{2.11}\label{loginnonclosed}
\end{align*}
 where $\vartheta_{\alpha,\beta,p}:= {[\beta+1-(\alpha-1)(p-1)]}/{p}$.
 In particular, $(\vartheta_{\alpha,\beta,p})^p$ is sharp if $N=n=p,\alpha=1, \beta\geq p-1$.
 \end{theorem}

 \begin{proof}
To begin with, set $\rho:=\log\left( \frac{d}{r} \right)$ and $c:=\alpha\left[(\alpha-1)(p-1)-\beta-1  \right]$.
 Then Lemma \ref{welaplaccompar} yields
\begin{align*}
-\Delta_{\mu,p}(c\rho^\alpha)
=|\alpha|^p        \frac{\rho^{(\alpha-1)(p-1)-1}}{r^p}       \left[(\alpha-1)(p-1)-\beta-1  \right]  \left[ -(\alpha-1)(p-1) +\rho(-p+1+r\Delta_\Psi r) \right]\geq0.
\end{align*}
On the other hand, for a small $\eta\in (0, \inj_o)$, Lemma \ref{volumecompacrions} (1) yields
\[
\frac{\sigma_o(r,y)}{r^{N-1}}\leq \frac{\sigma_o(\eta,y)}{\eta^{N-1}}\leq \frac{\max_{\zeta\in S_oM}\sigma(\eta,\zeta)}{\eta^{N-1}}=:C_{N}(\eta),\ \eta\leq r<i_y,\ y\in S_oM.\tag{2.12}\label{newvolumRiccN}
\]
Now it follows from (\ref{newvolumRiccN}), (\ref{log2use}) and (\ref{logconditnum}) that
 \begin{align*}
 \int_{M\setminus B_\eta(o)}\rho^{(\alpha-1)(p-1)}|\nabla\rho|^{p-1} d\mu=&\int_{S_oM}d\nu_o(y)\int^{i_y}_{\eta}\rho^{(\alpha-1)(p-1)} \frac{\sigma_o(r,y)}{r^{p-1}}dr\\
 \leq&   c_{n-1}C_N(\eta)\int^{\diam(M)}_\eta \left[\log\left( \frac{d}{r} \right)\right]^{(\alpha-1)(p-1)} r^{N-p}dr<\infty,
 \end{align*}
 which implies  $\rho^{(\alpha-1)(p-1)}|\nabla\rho|^{p-1}\in L^1_{\lo}({M_o})$. Similarly,  one has $\rho^{p+\beta}, \rho^\beta|\nabla \rho|^p\in L^1_{\lo}({M_o})$.
 Then (\ref{loginnonclosed}) follows from Lemma \ref{mainlemmforcr} immediately.

In the sequel, we show that (\ref{loginnonclosed}) is optimal if   $N=n=p, \alpha=1$ and $\beta\geq p-1$, in which case $\mathbf{Ric}_N=\mathbf{Ric}\geq 0$  and $\Psi=\text{const}$. It is not hard to check that $\rho$ satisfies the assumption of Definition \ref{DefDS} (i.e., $\rho/|\nabla\rho|\leq d$ and $\rho^{q(p+\beta)}\in L^1_{\loc}(M_o)$ for any $q> 1$).
And the standard volume comparison theorem yields
\[
d\mu=e^{-\Psi}{\dvol_g}\leq e^{-\Psi}r^{n-1}drd\nu_o(y), \ \forall\, 0<r<i_y,\ y\in S_oM.\tag{2.13}\label{standardvolumecomparison}
\]
By letting $s=\log \left( \frac{2d}{\inj_o} \right)$, one can easily check that
\[
{M_o}^-_s:=\rho^{-1}(-\infty,s]=M\backslash {B_{\inj_o/2}(o)},\ {M_o}^+_s:=\rho^{-1}(s,+\infty)=B_{\inj_o/2}(o)\backslash\{o\},
\]
 are non-empty subsets of $M_o$ with smooth boundaries. Let $c(\epsilon):= \frac{\beta+1+\epsilon}{p}$.
Then (\ref{standardvolumecomparison}) together with   (\ref{log2use}) yields
\begin{align*}
0<\int_{{M_o}^-_s} \rho^{c(\epsilon)p+\beta} |\nabla \rho|^p d\mu\leq c_{n-1}\,e^{-\Psi}\,\int^{\diam(M)}_{\inj_o/2}\rho^{c(\epsilon)p+\beta}r^{n-p-1}dr<\infty.
\end{align*}
On the other hand, since $-c(\epsilon)p+\beta=-1-\epsilon$,   (\ref{standardvolumecomparison}) and (\ref{log1use}) imply
\begin{align*}
0<\int_{{M_o}^+_s} \rho^{-c(\epsilon)p+\beta} |\nabla \rho|^p d\mu\leq c_{n-1}\,e^{-\Psi}\,\int_0^{\inj_o/2}\rho^{-1-\epsilon}r^{-1}dr<\infty.
\end{align*}
 Now for any $\epsilon\in (0,1)$, define a function
\begin{align*}v(x):=\left\{
\begin{array}{lll}
\left(\frac{\rho(x)}{s}\right)^{c(\epsilon)}, &&\text{ if  }x\in M\backslash B_{\inj_o/2}(o),\\
\\
\left(\frac{\rho(x)}{s}\right)^{-c(\epsilon/2)}, &&\text{ if  }x\in B_{\inj_o/2}(o)\backslash\{o\}.
\end{array}
\right.
\end{align*}
Since $\beta\geq p-1$ (i.e., $c(\epsilon)>1$), for any $\delta\in (0,1)$,  $v_\delta:=\max\{v-\delta,0\}$ is a globally Lipschitz function with compact support in $M_o$. Thus, the sharpness follows from Lemma \ref{secondlemmaforsharp}.
 \end{proof}

\begin{remark}\label{betaconstan} Firstly, provided that  $\mathbf{Ric}_N\geq (N-1)K>0$,  we can choose any number in $[\pi/\sqrt{K},+\infty)$ as $d$
(see for instance Ohta \cite[Theorem 2.4]{O1}). Secondly,
if $d>\diam(M)$, then $(\vartheta_{\alpha,\beta,p})^p$ is sharp when   $N=n=p$ and $\alpha=1$, i.e., the condition $\beta\geq p-1$ is unnecessary in this case.
\end{remark}

\begin{theorem}\label{rhardy1}
Let $(M,o,g,d\mu)$ be an $n$-dimensional closed pointed weighted Riemannian manifold with $\mathbf{Ric}_N\geq0$ for some $N\in [n,\infty)$.
Given $p,\beta\in \mathbb{R}$ with $p\in (1,N)\cup (N,\infty)$ and $\beta
< -N$,  we have
 \[
\int_{M_o} | \nabla u|^p r^{p+\beta} d\mu\geq \left( \frac{|N+\beta|}{p} \right)^p\int_{M_o} {|u|^p}r^{\beta}d\mu,\ \forall\, u\in C^\infty_0(M_o).\tag{2.14}\label{firsthardfordis}
\]
In particular, $\left( \frac{|N+\beta|}{p} \right)^p$ is sharp if $n=N$.
\end{theorem}
\begin{proof}(1) Set $\rho:=r|_{M_o}$, $c:=\alpha\left[(\alpha-1)(p-1)-\beta-1  \right]$ and  $\alpha:=(p-N)/(p-1)$. Obviously, one has
\[
\rho^{(\alpha-1)(p-1)}|\nabla\rho|^{p-1},\rho^\beta|\nabla \rho|^p,\rho^{p+\beta}\in L^1_{\lo}({M_o}).
\]
Moreover, Lemma \ref{welaplaccompar}  yields
\begin{align*}
-\Delta_{\mu,p} (c\rho^\alpha)=|\alpha|^pr^{(\alpha-1)(p-1)-1}(\beta+N)(1-N+r\Delta_\Psi r)\geq0.
\end{align*}
Hence, (\ref{firsthardfordis}) follows from Lemma \ref{mainlemmforcr} directly. Now we show (\ref{firsthardfordis}) is sharp if $N=n$, in which case we get $\Psi=\text{const}$ and (\ref{standardvolumecomparison}). Obviously, $\rho/|\nabla\rho|\leq \diam(M)$ and $\rho^{q(p+\beta)}\in L^1_{\loc}(M_o)$ for any $q> 1$.
Choose $s:=\inj_o/2>0$. Then ${M_o}^-_s:=\rho^{-1}(-\infty,s]=\overline{B_s(o)}\backslash\{o\}$ and ${M_o}^+_s:=\rho^{-1}(s,+\infty)=M\backslash \overline{B_s(o)}$ are non-empty subsets of $M_o$ with smooth boundaries. In particular, (\ref{standardvolumecomparison}) yields
\begin{align*}
0<\int_{{M_o}^-_s} \rho^{c(\epsilon)p+\beta} |\nabla \rho|^p d\mu
\leq\frac{c_{n-1}e^{-\Psi}\,s^{|n+\beta|+(n+\beta)+\epsilon}}{|n+\beta|+(n+\beta)+\epsilon}<\infty,
\end{align*}
where $c(\epsilon):=(|n+\beta|+\epsilon)/p$.
Since $M$ is compact, we have $0<\int_{{M_o}^+_s} \rho^{-c(\epsilon)p+\beta} |\nabla \rho|^p d\mu<\infty$. Now the sharpness follows from Lemma \ref{secondlemmaforsharp} immediately.
\end{proof}

The following example implies that if $N\neq n$, (\ref{firsthardfordis}) may be not optimal.
\begin{example}\label{notsharpexa}
Let $(\mathbb{S}^n,g)$ be an $n$-dimensional unit sphere  equipped with the canonical Riemannian metric and let  $o\in \mathbb{S}^n$ be a point.  Define a measure $d\mu=e^{-\Psi}\dvol_g$ on $\mathbb{S}^n$, where $\Psi$ is a constant.
Thus, for any $N\in [n,\infty)$, one has $\mathbf{Ric}_N=\mathbf{Ric}=n-1$.
Consequently, for any $p>N>n$, Theorem \ref{rhardy1} (also see Proposition  \ref{unitsphereexamp}) yields
\[
\int_{\mathbb{S}^n\backslash\{o\}}|\nabla u|^p {d\mu}\geq  \left(\frac{p-n}p \right)^p\int_{\mathbb{S}^n\backslash\{o\}}  \frac{|u|^p}{r^{p}} {d\mu}> \left(\frac{p-N}p \right)^p\int_{\mathbb{S}^n\backslash\{o\}}  \frac{|u|^p}{r^{p}} {d\mu}, \ \forall\,u\in C^\infty_0(M_o).
\]
Hence, $ \left(\frac{p-N}p \right)^p$  is not sharp.
\end{example}

By  a suitable modification, we obtain  the $\mathbf{Ric}_\infty$-versions of Theorems \ref{logriccnonclosed}-\ref{rhardy1}.
 \begin{theorem}\label{logriccnonclosedRicinf}
Let $(M,o,g,d\mu)$ be an $n$-dimensional closed pointed weighted Riemannian manifold with $\mathbf{Ric}_\infty\geq 0$  and $\diam(M)\leq d$.

\smallskip

(i) Suppose $\partial_r\Psi\geq -\lambda$ for some $\lambda\geq 0$.
Thus, for any $p,\beta\in \mathbb{R}$ and $\alpha\in \mathbb{R}\backslash\{0\}$ with
\[
 p\geq n+\lambda \diam(M), \  \log\left( \frac{d}{\diam(M)} \right)\left(n+\lambda \diam(M)-p\right)\leq (\alpha-1)(p-1)< \beta+1,
\]
we have
\begin{align*}
\int_{{M_o}} \left[\log\left(\frac{d}{r} \right)\right]^{p+\beta}|\nabla u|^p  {d\mu}\geq (\vartheta_{\alpha,\beta,p})^p\int_{{M_o}} \left[\log\left(\frac{d}{r} \right)\right]^{\beta} \frac{|u|^p}{r^p} {d\mu},\ \forall\,u\in C^\infty_0({M_o}),\tag{2.15}\label{Ricloginfty}
\end{align*}
 where $\vartheta_{\alpha,\beta,p}:= {[\beta+1-(\alpha-1)(p-1)]}/{p}$.
 In particular, $(\vartheta_{\alpha,\beta,p})^p$ is sharp if
 \[
 \lambda=0,\ n=p,\ \alpha=1,\ \beta\geq p-1.
 \]

 \smallskip

 (ii) Suppose $|\Psi|\leq k$ for some $k\geq 0$. Thus, for any $p,\beta\in \mathbb{R}$ and $\alpha\in \mathbb{R}\backslash\{0\}$ with
\[
p\geq n+4k, \  \log\left( \frac{d}{\diam(M)} \right)\left(n+4k-p\right)\leq (\alpha-1)(p-1)< \beta+1,
\]
we have
\begin{align*}
\int_{{M_o}} \left[\log\left(\frac{d}{r} \right)\right]^{p+\beta}|\nabla u|^p  {d\mu}\geq (\vartheta_{\alpha,\beta,p})^p\int_{{M_o}} \left[\log\left(\frac{d}{r} \right)\right]^{\beta} \frac{|u|^p}{r^p} {d\mu},\ \forall\,u\in C^\infty_0({M_o}),
\end{align*}
where $\vartheta_{\alpha,\beta,p}$ is defined as above.
In particular, $(\vartheta_{\alpha,\beta,p})^p$ is sharp if $ k=0$, $n=p$, $\alpha=1$ and $\beta\geq p-1$.

 \end{theorem}
\begin{proof}  Set $\rho:=\log\left( \frac{d}{r} \right)$ and $c:=\alpha\left[(\alpha-1)(p-1)-\beta-1  \right]$. We now show (i). It follows from Lemma \ref{welaplaccompar} (2) that $-\Delta_{\mu,p}(c\rho^\alpha)\geq 0$. And Lemma \ref{volumecompacrions} (2) furnishes
\[
d\mu\leq e^{-\Psi(o)+\lambda \diam(M)}r^{n-1}drd\nu_o(y), \ \forall\, 0<r<i_y,\ y\in S_oM,\tag{2.16}\label{standardvolumecomparison2}
\]
which together with (\ref{log2use}) implies $\rho^{(\alpha-1)(p-1)}|\nabla\rho|^{p-1}, \rho^{p+\beta}, \rho^\beta|\nabla \rho|^p\in L^1_{\lo}({M_o}).$
Thus, (\ref{Ricloginfty}) follows from Lemma \ref{mainlemmforcr} directly. The proof of the sharpness is analogous to that  of Theorem \ref{logriccnonclosed}.

\smallskip

(ii) Lemma \ref{welaplaccompar} (3) implies $-\Delta_{\mu,p}(c\rho^\alpha)\geq 0$. For a small $\eta\in (0, \inj_o)$, Lemma \ref{volumecompacrions} (3) yields
\[
\frac{\sigma(r,y)}{r^{n+4k-1}}\leq \frac{\sigma(\eta,y)}{\eta^{n+4k-1}}\leq \frac{\max_{\zeta\in S_oM}\sigma(\eta,\zeta)}{\eta^{n+4k-1}}=:C_{n+4k}(\eta),\ \eta\leq r<i_y,\ y\in S_oM.\tag{2.17}\label{newvolumRiccNinfty}
\]
From (\ref{newvolumRiccNinfty}) and (\ref{log2use}), we derive
 \begin{align*}
 \int_{M\setminus B_\eta(o)}\rho^{(\alpha-1)(p-1)}|\nabla\rho|^{p-1} d\mu&\leq   c_{n-1}C_{n+4k}(\eta)\int^{\diam(M)}_\eta \left[\log\left( \frac{d}{r} \right)\right]^{(\alpha-1)(p-1)} r^{n+4k-p}dr<\infty,
 \end{align*}
 which implies  $\rho^{(\alpha-1)(p-1)}|\nabla\rho|^{p-1}\in L^1_{\lo}({M_o})$. Similarly,  one has $\rho^{p+\beta}, \rho^\beta|\nabla \rho|^p\in L^1_{\lo}({M_o})$.
The remainder of the argument is the same as (i) and therefore, we omit it.
\end{proof}

Using the same arguments   as in the proofs  of Theorems \ref{rhardy1}-\ref{logriccnonclosedRicinf}, one can easily   carry out  the proof of the following result.

\begin{theorem}\label{rhardyRicinf}
Let $(M,o,g,d\mu)$ be an $n$-dimensional closed pointed weighted Riemannian manifold with  $\mathbf{Ric}_\infty\geq 0$.

\smallskip

(i)  Suppose $\partial_r\Psi\geq -\lambda$ for some $\lambda\geq 0$. Thus, for any
  $p,\beta\in \mathbb{R}$ and $\alpha\in \mathbb{R}\backslash\{0\}$ with
\[
p> 1, \ \beta+1<(\alpha-1)(p-1)\leq -\left[n-1+\lambda \diam(M)\right],
\]
we have
 \[
\int_{M_o} | \nabla u|^p r^{p+\beta}d\mu\geq (\vartheta_{\alpha,\beta,p})^p\int_{M_o} {|u|^p}r^{\beta}d\mu,\ \forall\, u\in C^\infty_0(M_o),
\]
 where $\vartheta_{\alpha,\beta,p}:= {[(\alpha-1)(p-1)-\beta-1]}/{p}$.
 In particular, $(\vartheta_{\alpha,\beta,p})^p$ is sharp if $\lambda=0$ and $\alpha=\frac{p-n}{p-1}\neq0$.

\smallskip

(ii)  Suppose $|\Psi|\leq k$ for some $k\geq 0$.
Thus, for any $p,\beta\in \mathbb{R}$ and $\alpha\in \mathbb{R}\backslash\{0\}$ with
\[
p> 1, \ \beta+1<(\alpha-1)(p-1)\leq -(n+4k-1),
\]
we have
 \[
\int_{M_o} | \nabla u|^p r^{p+\beta} d\mu\geq (\vartheta_{\alpha,\beta,p})^p\int_{M_o} {|u|^p}r^{\beta}d\mu,\ \forall\, u\in C^\infty_0(M_o),
\]
where $\vartheta_{\alpha,\beta,p}$ is defined as above. In particular, $(\vartheta_{\alpha,\beta,p})^p$ is sharp if $k=0$  and $\alpha=\frac{p-n}{p-1}\neq0$.
\end{theorem}

In the rest of this section, we study the Brezis-V\'azquez improvement.   Inspired by D'Ambrosio et al. \cite[Theorem 4.1]{DD}, we obtain the following lemma.
\begin{lemma}\label{leBV} Let $(M,o,g,d\mu)$ be an $n$-dimensional closed pointed weighted Riemannian manifold. Given $p\geq 2$, suppose that $\rho\in W^{1,p}_{\loc}(M_o)$ is a nonnegative function such that  $|\nabla\rho|\neq0$ a.e. and $-\Delta_{\mu,p}\rho\geq 0$ in the weak sense.
Set
\[
\Theta:=\inf_{u\in C^\infty_0(M_o)\backslash\{0\}}\frac{\int_{M_o} \rho |\nabla\rho|^{p-2}\, | \nabla u|^2d\mu}{\int_{M_o}\rho |\nabla\rho|^{p-2}\,|u|^2 d\mu}.
\]
Then for any $u\in C^\infty_0(M_o)$, we have
\begin{align*}
\int_{M_o}| \nabla u|^p {d\mu}\geq \left(\frac{p-1}p \right)^p\int_{M_o}  \frac{|u|^p}{\rho^p}|\nabla\rho|^p {d\mu}+\frac{2\Theta}{p}\left(\frac{p-1}p \right)^{p-2}\int_{M_o}\frac{|u|^p}{\rho^{p-2}} |\nabla\rho|^{p-2}{d\mu}.\tag{2.18}\label{impgeneral}
\end{align*}
\end{lemma}
\begin{proof}The assumption together with the H\"older inequality implies $\rho |\nabla\rho|^{p-2}\in L^1_{\loc}(M_o)$ and hence, $\Theta$ is well-defined.

Given $u\in C^\infty_0(M_o)$ and $\varepsilon\in (0,1)$, define $\rho_\varepsilon:=\rho+\varepsilon$,  $\gamma:=\frac{p-1}p$ and $v:=u/\rho^\gamma_\varepsilon$. A direct calculation yields
\begin{align*}
|\nabla u|^2=\gamma^2 v^2 {\rho_\varepsilon}^{2(\gamma-1)}|\nabla\rho_\varepsilon|^2+ {{\rho_\varepsilon}^{2\gamma}}|\nabla v|^2+2\gamma v{\rho_\varepsilon}^{2\gamma-1} g(\nabla{\rho_\varepsilon},\nabla v):=\alpha-\beta,
\end{align*}
where $\alpha:=\gamma^2 v^2 {\rho_\varepsilon}^{2(\gamma-1)}|{\nabla\rho_\varepsilon}|^2$ and $\beta:=- {{\rho_\varepsilon}^{2\gamma}}|\nabla v|^2-2\gamma v{\rho_\varepsilon}^{2\gamma-1}g(\nabla{\rho_\varepsilon},\nabla v)$.

Recall $(\alpha-\beta)^s
\geq \alpha^s-s\beta\alpha^{s-1}$ if  $\alpha>0$, $\alpha>\beta$ and $s\geq 1$. By letting $s=p/2$, we have
\begin{align*}
|\nabla u|^p\geq
\gamma^p \frac{|u|^p}{{\rho_\varepsilon}^p}|\nabla{\rho_\varepsilon}|^p+\gamma^{p-1}  |\nabla{\rho_\varepsilon}|^{p-2}g({\nabla\rho_\varepsilon}, \nabla|v|^p)+\frac{2\gamma^{p-2}}{p}{\rho_\varepsilon}  |\nabla{\rho_\varepsilon}|^{p-2}\left| \nabla|v|^{\frac{p}2}  \right|^2,
\end{align*}
which together with $-\Delta_{\mu,p} \rho\geq 0$ yields
\begin{align*}
&\int_{M_o}|\nabla u|^p {d\mu}-\gamma^p\int_{M_o}  \frac{|u|^p}{{\rho_\varepsilon}^p}|\nabla \rho|^p {d\mu}\geq\frac{2\gamma^{p-2}}{p}\int_{M_o}{\rho}\,  |\nabla{\rho}|^{p-2}\,\left( \nabla|v|^{\frac{p}2}  \right)^2{d\mu}\\
\geq &\frac{2\Theta \gamma^{p-2}}{p }\int_{M_o} {\rho}|\nabla{\rho}|^{p-2}|v|^p {d\mu}= \frac{2\Theta \gamma^{p-2}}{p} \int_{M_o} \frac{|u|^p}{\rho_\varepsilon^{p-1}}\rho |\nabla{\rho}|^{p-2} {d\mu}.\tag{2.19}\label{converbv}
\end{align*}
It follows from   D'Ambrosio et al. \cite[Theorem 2.1]{DD} that $|\nabla\rho|/\rho\in L^p_{\lo}(M_o)$. Consequently, we have
\begin{align*}
\left|  \frac{|u|^p}{{\rho_\varepsilon}^p}|\nabla{\rho}|^p  \right|\leq \frac{|u|^p}{{\rho}^p}|\nabla{\rho}|^p\in L^1(M_o),\ \left|\frac{|u|^p}{\rho_\varepsilon^{p-1}}\rho  |\nabla{\rho}|^{p-2}\right|\leq \frac{|u|^p}{\rho^{p-2}}  |\nabla{\rho}|^{p-2}\in L^1(M_o),
\end{align*}
which together with Lebesgue's dominated convergence theorem and (\ref{converbv}) yields (\ref{impgeneral}).\end{proof}

\begin{corollary}\label{cannotachconsta}
 Let $(M,o,g,d\mu)$ be an $n$-dimensional closed pointed weighted Riemannian manifold. Suppose that $\rho\in C^1(M_o)\cap W^{1,p}_{\lo}(M_o)$  is a nonnegative function such that $\rho\neq0$ a.e.,  $\rho/|\nabla\rho|\leq C$ for some positive constant $C$ and $-\Delta_{\mu,p}\rho\geq 0$ in the weak sense. Then for any $p\geq 2$,
 \[
\int_{M_o}| \nabla u|^p {d\mu}\geq \left(\frac{p-1}p \right)^p\int_{M_o}  \frac{|u|^p}{\rho^p}|\nabla\rho|^p {d\mu}, \ \forall\,u\in D^{1,p}(M_o,\rho^0).
\]
In particular, if $\left(\frac{p-1}p \right)^p$ is sharp and $\rho^{\frac{p-1}p}\notin D^{1,p}(M_o,\rho^0)$, then $\left(\frac{p-1}p \right)^p$ cannot be achieved.
\end{corollary}
\begin{proof}
Given $u\in D^{1,p}(M_o,\rho^0)$, there exists a sequence $u_j\in C^\infty_0(M_o)$ converges to $u$ with respect to $\|\cdot\|_D$. By passing a subsequence, we can assume that  $\nabla u_j$ converges to $\nabla u$ pointwise a.e. And Lemma \ref{controlnormandspace} yields $u_j\rightarrow u$ with respect to $\|\cdot\|_L:=(\int_{M_o}|\cdot|^p |\nabla\rho|^p/\rho^p d\mu)^{1/p}$. Since $\rho/|\nabla\rho|\leq C$, we have
\[
\left(\int_{M_o}|u_j-u|^p d\mu\right)^{\frac1p}\leq C \|u_j-u\|_L\rightarrow 0, \text{ as }j\rightarrow \infty.
\]
Hence, by passing a subsequence, we can assume that $u_j$ converges to $u$ pointwise a.e. Given $\varepsilon>0$, set $\rho_\varepsilon:=\rho+\varepsilon$,
 $v_{j,\varepsilon}:=u_j/\rho_\varepsilon^\gamma$ and $v_\varepsilon:=u/\rho_\varepsilon^\gamma$, where $\gamma:=(p-1)/p$. Thus,   Lieb and Loss
\cite[Theorem 6.17]{LL} implies that  $\nabla|v_{j,\varepsilon}|^{p/2}$ converges to $\nabla|v_\varepsilon|^{p/2}$ pointwise a.e.

Repeating the same argument as in the proof of Lemma \ref{leBV} and  Fatou's lemma, we have
\begin{align*}
&\int_{M_o}|\nabla u|^p {d\mu}-\gamma^p\int_{M_o}  \frac{|u|^p}{{\rho_\varepsilon}^p}|\nabla \rho|^p {d\mu}
=\underset{j\rightarrow \infty}{\lim\inf}\left[\int_{M_o}|\nabla u_j|^p {d\mu}-\gamma^p\int_{M_o}  \frac{|u_j|^p}{{\rho_\varepsilon}^p}|\nabla \rho|^p {d\mu}\right]\\
\geq& \underset{j\rightarrow \infty}{\lim\inf}\frac{2\gamma^{p-2}}{p}\int_{M_o}{\rho}\,  |\nabla{\rho}|^{p-2}\,\left( \nabla|v_{j,\varepsilon}|^{\frac{p}2}  \right)^2{d\mu}\geq \frac{2\gamma^{p-2}}{p}\int_{M_o}{\rho}\,  |\nabla{\rho}|^{p-2}\,\left( \nabla|v_\varepsilon|^{\frac{p}2}  \right)^2{d\mu}.
\end{align*}
Using Fatou's lemma and  Lebesgue's dominated convergence theorem again, we have
\begin{align*}
\int_{M_o}|\nabla u|^p {d\mu}-\gamma^p\int_{M_o}  \frac{|u|^p}{{\rho}^p}|\nabla \rho|^p {d\mu}\geq\frac{2\gamma^{p-2}}{p}\int_{M_o}{\rho}\,  |\nabla{\rho}|^{p-2}\,\left( \nabla|v|^{\frac{p}2}  \right)^2{d\mu},
\end{align*}
where $v=u/\rho^\gamma$.
Note that $|u|\in D^{1,p}(M_0,\rho^0)$ (cf. Lieb et al.
\cite[Theorem 6.17]{LL}). So, if $\rho^{\gamma}\notin D^{1,p}(M_o,\rho^0)$, then the above inequality
implies the nonexistence of minimizers in $D^{1,p}(M_o,\rho^0)$.
\end{proof}

By letting $\rho:=r^{\frac{p-n}{p-1}}$, one can derive the following corollary  from Lemma \ref{leBV} and Corollary \ref{cannotachconsta} immediately.
\begin{corollary}\label{rcoro1} Let $(M,o,g,d\mu)$ be an $n$-dimensional closed pointed weighted Riemannian manifold with $\mathbf{Ric}_\infty\geq 0$ and $\partial_r\Psi\geq0$.
 Define $\Theta_r$ as
\[
\Theta_r:=\inf_{u\in C^\infty_0(M_o)\backslash\{0\}}\frac{\int_{M_o} r^{2-n}\, |\nabla u|^2 {d\mu}}{\int_{M_o}r^{2-n}\,|u|^2 {d\mu}}.\tag{2.20}\label{theta1r}
\]
Given $p> n$, for any $u\in C^\infty_0(M_o)$, we have
\begin{align*}
\int_{M_o}|\nabla u|^p {d\mu}\geq& \left(\frac{p-n}p \right)^p\int_{M_o}  \frac{|u|^p}{r^{p}} {d\mu}+\frac{2\Theta_r}{p}\left(\frac{p-n}p \right)^{p-2}\int_{M_o}\frac{|u|^p}{r^{p-2}}{d\mu}.
\end{align*}
In particular, $\left(\frac{p-n}p \right)^p$ is sharp but not achieved.
\end{corollary}

\begin{remark}It is usually hard to determine whether $\Theta_r$ is positive. Now we turn to consider a weaker case.
For any $\epsilon\in(0,\inj_o/2)$, set $M_\epsilon:=M\backslash \overline{B_\epsilon(o)}$ and
\[
\Theta_{r,\epsilon}:=\inf_{u\in C^\infty_0(M_\epsilon)\backslash\{0\}}\frac{\int_{M_\epsilon} r^{2-n}\, |\nabla u|^2 {d\mu}}{\int_{M_\epsilon}r^{2-n}\,|u|^2 {d\mu}}.
\]
Then $\Theta_{r,\epsilon}>0$ follows from the standard theory of the Dirichlet eigenvalue problem. Moreover, by repeating the same argument as before, one has
\begin{align*}
\int_{M_\epsilon}|\nabla u|^p {d\mu}\geq& \left(\frac{p-n}p \right)^p\int_{M_\epsilon}  \frac{|u|^p}{r^{p}} {d\mu}+\frac{2\Theta_{r,\epsilon}}{p}\left(\frac{p-n}p \right)^{p-2}\int_{M_\epsilon}\frac{|u|^p}{r^{p-2}}{d\mu},\ \forall \, u\in C^\infty_0(M_\epsilon).
\end{align*}
In particular, $\left(\frac{p-n}p \right)^p$ is sharp but not achieved. Also refer to Yang et al. \cite[Lemma 4.3]{YSK} for the complete noncompact case.

\end{remark}

\section{Hardy type inequalities on closed manifolds}\label{distanceHardy}

\begin{definition}\label{newst2defi}
Let $(M,o,g,d\mu)$ be an $n$-dimensional closed pointed weighted Riemannian manifold. Let $\Omega$ denote either $M$ or $M_o$.
Given $p\in (1,\infty)$ and $\beta\in \mathbb{R}$, let $\rho$ be a nonconstant nonnegative measurable function on $M$ such that $\rho\neq0$ a.e. on $M$ and  $\rho^{{(p+\beta)}/{(1-p)}},\rho^{q(p+\beta)}\in L^1(M)$ for some $q\in (1,\infty)$.
Now we define a weighted $L^p$-norm on $C^\infty_0(\Omega)$ by
\[
\|u\|_{p,\beta}:=\left( \int_{\Omega}|u|^p \rho^{p+\beta} d\mu\right)^{\frac1p}.
\]
And the {\it weighted Sobolev space} $W^{1,p}(\Omega, \rho^{p+\beta})$  is defined as the completion of $C_0^\infty(\Omega)$ with respect to
\[
\|u\|_{{1,p,\beta}}:=\|u\|_{{p,\beta}}+\||\nabla u|\|_{p,\beta}.\tag{3.1}\label{secondnormdefi}
\]
\end{definition}


\begin{lemma}\label{thsob}Let $(M,o,g,d\mu),p,\beta$ and $\rho$ be as in Definition \ref{newst2defi}.
If $u\in W^{1,p}({M}, \rho^{p+\beta})\cap C({M})$ with $u(o)=0$, then $u|_{M_o}\in W^{1,p}({M_o}, {\rho^{p+\beta}})$.
\end{lemma}
\begin{proof}According to Lemma \ref{maxminle}, $u_\pm\in  W^{1,p}({M}, \rho^{p+\beta})\cap C({M})$ with $u_\pm(o)=0$, $u_\pm\geq0$ and $u=u_+-u_-$. Thus,
without loss of generality, we may assume $u\geq 0$ (i.e. $u=u_+$).
Given $\delta\in (0,1)$, set $u_\delta(x):=\max\{u-\delta,0\}$. Since  $u$ is continuous with $u(o)=0$, there exists a small $\eta>0$ such that $u_\delta=0$ in $B_\eta(o)$. Since $u_\delta\in W^{1,p}(M,\rho^{p+\beta})$ (see Lemma \ref{maxminle}), a standard argument by cut-off functions then yields
 $u_\delta|_{{M_o}} \in W^{1,p}({M_o},\rho^{p+\beta})$.
On the other hand,
the dominated convergence theorem furnishes
\begin{align*}
&\left\|u|_{{M_o}}-u_\delta|_{{M_o}}\right\|_{1,p,\beta}=\left\|  u-  u_\delta   \right\|_{1,p,\beta}=\left(\int_{{M}}|u-u_\delta|^p\rho^{p+\beta} {d\mu}\right)^{\frac1p} + \left( \int_{{M}} |\nabla (u-u_\delta)|^p \rho^{p+\beta} {d\mu}\right)^{\frac1p}\\
\leq & \left(\delta^p \int_M \rho^{p+\beta} {d\mu}+\int_{{M}}1_{\{0\leq u\leq \delta\}}|u|^p\rho^{p+\beta} {d\mu}\right)^{\frac1p} + \left(\int_{{M}} 1_{\{0\leq u\leq \delta\}} |\nabla u|^p \rho^{p+\beta} {d\mu}\right)^{\frac1p}\rightarrow 0, \text{ as }\delta\rightarrow 0^+.
\end{align*}
Hence, $u|_{{M_o}}\in  W^{1,p}({M_o},\rho^{p+\beta})$.
\end{proof}


In the sequel, we set $C^\infty(M,o):=\{u\in C^\infty(M):\,u(o)=0\}$.

 \begin{theorem}\label{logriccnonclosed2}
Let $(M,o,g)$ be an $n$-dimensional closed pointed Riemannian manifold with $\mathbf{Ric}_N\geq 0$ for some $N\in[n,\infty)$.  Suppose that $d,p,\beta\in \mathbb{R}$ and $\alpha\in \mathbb{R}\backslash\{0\}$ satisfy
\[
d>\diam(M),\ p\geq N, \ \log\left( \frac{d}{\diam(M)} \right)(N-p)\leq (\alpha-1)(p-1)< \beta+1.
\]
Then we have
\begin{align*}
\int_{{M}} \left[\log\left(\frac{d}{r} \right)\right]^{p+\beta}|\nabla u|^p d\mu\geq (\vartheta_{\alpha,\beta,p})^p\int_{{M}} \left[\log\left(\frac{d}{r} \right)\right]^{\beta} \frac{|u|^p}{r^p} {d\mu},\ \forall\,u\in C^\infty({M,o}),\tag{3.1}\label{logfirstonM}
\end{align*}
 where $\vartheta_{\alpha,\beta,p}:= {[\beta+1-(\alpha-1)(p-1)]}/{p}$.

 In particular, if $p=N=n$ and $\alpha=1$, then $(\vartheta_{\alpha,\beta,p})^p=\left(  \frac{\beta+1}{n}\right)^n$ is sharp with respect to $C^\infty(M,o)$, i.e.,
\[
(\vartheta_{\alpha,\beta,p})^p=\inf_{u\in  C^\infty(M,o)\backslash\{0\}}\frac{\int_{{M}} \left[\log\left(\frac{d}{r} \right)\right]^{p+\beta}|\nabla u|^p d \mu}{\left[\log\left(\frac{d}{r} \right)\right]^{\beta} \frac{|u|^p}{r^p} {d\mu}}.
\]
 \end{theorem}
\begin{proof}Let $\rho:=\log\left( d/r \right)$.
Since $d>\diam(M)$ and $\sigma_o(r,y)\sim e^{-\Psi(o)}r^{n-1}$ for small $r$, Lemma \ref{interglemaGa} implies $\rho^{{(p+\beta)}/{(1-p)}},\rho^{q(p+\beta)}\in L^1(M)$ for any $q>1$.  In view of the proof of Theorem \ref{logriccnonclosed},
$\rho$ satisfies the assumptions of Lemma \ref{secondlemmaforsharp} and Definition \ref{newst2defi}.

Given $u\in C^\infty(M,o)$, Lemma \ref{thsob} yields  a sequence $u_j\in C^\infty_0(M_o)$ with $\|u_j-u|_{M_o}\|_{1,p,\beta}\rightarrow 0$. Since $\rho\neq0$ on $M_o$, by passing a subsequence, we can assume that
 $u_j$ converges to $u|_{M_o}$ pointwise a.e. on $M_o$.
Now Theorem \ref{logriccnonclosed} together with Fatou's lemma furnishes
\begin{align*}
 &(\vartheta_{\alpha,\beta,p})^p\int_{{M}} |u|^p\rho^{\beta}|\nabla\rho|^p
 {d\mu}=(\vartheta_{\alpha,\beta,p})^p\int_{{M_o}}|u|^p\rho^{\beta}|\nabla\rho|^p {d\mu}\\
  \leq& (\vartheta_{\alpha,\beta,p})^p\underset{j\rightarrow\infty}{\lim\inf}\int_{{M_o}} |u_j|^p\rho^{\beta}|\nabla\rho|^p{d\mu} \leq \underset{j\rightarrow\infty}{\lim\inf}\int_{{M_o}}|\nabla u_j|^p \rho^{p+\beta} d\mu=\int_{{M}} |\nabla u|^p\rho^{p+\beta} d\mu,
\end{align*}
where the last equality follows from $\|u_j-u|_{M_o}\|_{D}\rightarrow 0$. Hence,  (\ref{logfirstonM}) follows.

In order to prove the sharpness, let $v\in D^{1,p}(M_o,\rho^{p+\beta})$ be defined as in Lemma \ref{secondlemmaforsharp} (also see Remark \ref{betaconstan}). Hence, there exists  a sequence $v_j\in C^\infty_0(M_o)$ with $\|v_j- v\|_{D}\rightarrow0$. On one hand, by the zero extension, $v_j$ can be viewed as a function in $C^\infty(M,o)$ and $v$ can defined on $M$. On the other hand, since $o$ is of zero measure,  Lemma \ref{controlnormandspace} then yields
\[
\int_{M} |v_j|^p \rho^\beta |\nabla\rho|^pd\mu\rightarrow \int_{M} |v|^p \rho^\beta |\nabla\rho|^pd\mu,\ \ \ \int_{M} |\nabla v_j|^p\rho^{p+\beta}d\mu\rightarrow \int_{M} |\nabla v|^p\rho^{p+\beta}d\mu.
\]
Thus,  the  remainder of the proof is the same as the one of Lemma \ref{secondlemmaforsharp}.
\end{proof}

\begin{proof}[Proof of Theorem \ref{reverRicinfty}]
Let $\rho:=r$. Since $p+\beta>-n$ and $\sigma_o(r,y)\sim e^{-\Psi(o)}r^{n-1}$ for small $r$, we have $\rho^{{(p+\beta)}/{(1-p)}},\rho^{q(p+\beta)}\in L^1(M)$ for any $q>1$ with $q(p+\beta)>-n$. It follows from the proof of Theorem \ref{rhardy1} that $\rho$ satisfies the  assumptions of Lemma \ref{secondlemmaforsharp} and Definition \ref{newst2defi}. Using Theorem \ref{rhardy1} and the same argument as in the proof of Theorem \ref{logriccnonclosed2}, one can easily complete the proof of this theorem.
\end{proof}

Now we consider  Hardy type inequalities in the context of the $\infty$-Ricci curvature.
 \begin{theorem}\label{logriccnonRicinf2closed}
Let $(M,o,g,d\mu)$ be an $n$-dimensional closed pointed weighted Riemannian manifold with  $\mathbf{Ric}_\infty\geq 0$ and $\diam(M)<d$.

\smallskip

(i) Suppose $\partial_r\Psi\geq -\lambda$ for some $\lambda\geq 0$.
Thus, for any $p,\beta\in \mathbb{R}$ and $\alpha\in \mathbb{R}\backslash\{0\}$ with
\[
   p\geq n+\lambda \diam(M), \  \log\left( \frac{d}{\diam(M)} \right)\left(n+\lambda \diam(M)-p\right)\leq (\alpha-1)(p-1)< \beta+1,
\]
we have
\begin{align*}
\int_{{M}} \left[\log\left(\frac{d}{r} \right)\right]^{p+\beta}|\nabla u|^p  {d\mu}\geq (\vartheta_{\alpha,\beta,p})^p\int_{{M}} \left[\log\left(\frac{d}{r} \right)\right]^{\beta} \frac{|u|^p}{r^p} {d\mu},\ \forall\,u\in C^\infty({M},o),
\end{align*}
 where $\vartheta_{\alpha,\beta,p}:= {[\beta+1-(\alpha-1)(p-1)]}/{p}$.
 In particular, $(\vartheta_{\alpha,\beta,p})^p$ is sharp  with respect to $C^\infty(M,o)$ if
 \[
 \lambda=0,\ n=p,\ \alpha=1.
 \]

 \smallskip

 (ii) Suppose $|\Psi|\leq k$ for some $k\geq 0$.
Thus, for any $p,\beta\in \mathbb{R}$ and $\alpha\in \mathbb{R}\backslash\{0\}$  with
\[
p\geq n+4k, \  \log\left( \frac{d}{\diam(M)} \right)\left(n+4k-p\right)\leq (\alpha-1)(p-1)< \beta+1,
\]
we have
\begin{align*}
\int_{{M}} \left[\log\left(\frac{d}{r} \right)\right]^{p+\beta}|\nabla u|^p d\mu\geq (\vartheta_{\alpha,\beta,p})^p\int_{{M}} \left[\log\left(\frac{d}{r} \right)\right]^{\beta} \frac{|u|^p}{r^p} d\mu,\ \forall\,u\in C^\infty({M},o),
\end{align*}
 where $\vartheta_{\alpha,\beta,p}$ is defined as above. In particular, $(\vartheta_{\alpha,\beta,p})^p$ is sharp if $k=0$, $n=p$ and $\alpha=1$.
 \end{theorem}
\begin{proof}
Let $\rho:=\log(d/r)$. By Theorem \ref{logriccnonclosedRicinf}, the proof is almost the same as that of Theorem \ref{logriccnonclosed2}.
\end{proof}

\begin{proof}[Proof of Theorem \ref{secondrpesenthe}]
Let $\rho:=r$.
In view of Theorem \ref{rhardyRicinf}, one can prove  the statements by a suitable modification to the proof of Theorem \ref{logriccnonclosed2}.
\end{proof}

Furthermore, we  study the Brezis-V\'azquez improvement.

\begin{theorem}\label{thridthe} Let $(M,g,d\mu)$ be an $n$-dimensional closed weighted Riemannian manifold with $\mathbf{Ric}_\infty\geq 0$ and $\partial_r\Psi\geq 0$.
Given $p> n$, for any $u\in C^\infty(M,o)$, we have
\begin{align*}
\int_{M}|\nabla u|^p d\mu\geq& \left(\frac{p-n}p \right)^p\int_{M}  \frac{|u|^p}{r^{p}} d\mu+\frac{2\,\Xi_r}{p}\left(\frac{p-n}p \right)^{p-2}\int_{M}\frac{|u|^p}{r^{p-2}}d\mu,
\end{align*}
where
\[
\Xi_r:=\inf_{u\in C^\infty(M,o)\backslash\{0\}}\frac{\int_{M} r^{2-n}\, |\nabla u|^2 d\mu}{\int_{M}r^{2-n}\,|u|^2 d\mu}.
\]
In particular, $\left(\frac{p-n}p \right)^p$ is sharp but not achieved and $\Xi_r=\Theta_r$, where $\Theta_r$ is defined in (\ref{theta1r}).
\end{theorem}

\begin{proof} Set $\rho:=r^{\frac{p-n}{p-1}}$.  Corollary \ref{rcoro1} together with the same argument as in the proof of Theorem \ref{logriccnonclosed2}
yields
\begin{align*}
\int_{M}|\nabla u|^p d\mu\geq& \left(\frac{p-n}p \right)^p\int_{M}  \frac{|u|^p}{r^{p}} d\mu+\frac{2\Theta_r}{p}\left(\frac{p-n}p \right)^{p-2}\int_{M}\frac{|u|^p}{r^{p-2}}d\mu,\ \forall\,u\in C^\infty(M,o).
\end{align*}
 In particular, it follows from Theorem \ref{secondrpesenthe} that $\left[({p-n})/p \right]^p$ is sharp. We claim that  $\left[({p-n})/p \right]^p$ cannot be achieved.
If not, there would be a function $u\in C^\infty(M,o)\backslash\{0\}$ such that
\begin{align*}
\int_{M_o}|\nabla u|^p d\mu=\int_{M}|\nabla u|^p d\mu= \left(\frac{p-n}p \right)^p\int_{M}  \frac{|u|^p}{r^{p}} d\mu=\left(\frac{p-n}p \right)^p\int_{M_o}  \frac{|u|^p}{r^{p}} d\mu.\tag{3.2}\label{new3.3contar}
\end{align*}
Lemma \ref{thsob} and Remark \ref{newimoprremar} imply $u|_{M_o}\in W^{1,p}(M_o,\rho^0)=D^{1,p}(M_o,\rho^0)$. But  (\ref{new3.3contar}) then contradicts with Corollary \ref{cannotachconsta}, since $r^{\frac{p-n}p}=\rho^{\frac{p-1}{p}}\notin D^{1,p}(M_o,\rho^0)$.

We concludes the proof by showing $\Xi_r=\Theta_r$.
It suffices to show $\Xi_r\geq\Theta_r$. Given $u\in C^\infty(M,o)$, Lemma \ref{thsob} yields $u|_{M_o}\in W^{1,p}(M_o,\rho^0)$ and hence, there is a sequence $u_j\in C^\infty_0(M_o)$ with
\[
\left(\int_M |u_j-u|^pd\mu\right)^{\frac1p}+\left(\int_M |\nabla u_j- \nabla u|^pd\mu\right)^{\frac1p}=\|u_j-u|_{M_o}\|_{1,p,\beta}\rightarrow 0,\text{  as }j\rightarrow \infty.\tag{3.3}\label{exprectring}
\]
The H\"older inequality together with (\ref{standardvolumecomparison2}) furnishes
\begin{align*}
&\int_{M}|f|^2 r^{2-n} d\mu\leq \left(  \int_M |f|^p d\mu \right)^{\frac2p}\left(  \int_M r^{\frac{p(2-n)}{p-2}} d\mu\right)^{\frac{p-2}{p}}\\
\leq& \left(  \int_M |f|^p d\mu \right)^{\frac2p}\left(c_{n-1} e^{-\Psi(o)} \int_0^{\diam(M)} r^{\frac{p(2-n)}{p-2}+n-1} dr\right)^{\frac{p-2}{p}}<\infty, \ \forall\, f\in C^\infty(M,o),
\end{align*}
which together with (\ref{exprectring}) yields
\[
\int_{M} r^{2-n}\, |\nabla u|^2 d\mu=\lim_{j\rightarrow\infty}\int_{M} r^{2-n}\, |\nabla u_j|^2 d\mu, \ \int_{M}r^{2-n}\,|u|^2 d\mu=\lim_{j\rightarrow\infty}\int_{M}r^{2-n}\,|u_j|^2 d\mu.
\]
Therefore, $\Xi_r\geq {\Theta}_r$.
\end{proof}

\begin{proof}[Proof of Theorem \ref{newharyineqtype}]
Note that $u-u(o)\in C^\infty(M,o)$, for any $u\in C^\infty(M)$. Hence, Theorem \ref{newharyineqtype} is a direct consequence of Theorem \ref{thridthe}.
\end{proof}

Moreover, we have a uncertainty principle inequality by Theorem \ref{thridthe} and the H\"older inequality.
\begin{corollary}
 Let $(M,o,g,d\mu)$ be an $n$-dimensional closed pointed weighted Riemannian manifold with  $\mathbf{Ric}_\infty\geq 0$ and $\partial_r\Psi\geq 0$. For any $p> n$, we have
\[
\left(\int_M r^q |u|^q {d\mu}\right)^{\frac1q}\left( \int_M   |\nabla u|^p  {d\mu} \right)^{\frac1p} \geq \left(\frac{p-n}p \right)\int_{M} u^2 {d\mu},\ \forall\,u\in C^\infty(M,o),
\]
where $1/p+1/q=1$.
\end{corollary}

\section{Hardy type inequalities on the unit sphere}\label{exapmlesec}
Let $(\mathbb{S}^n,g)$ be an $n$-dimensional unit sphere  equipped with the canonical Riemannian metric and let  $o\in \mathbb{S}^n$ be a point. Set $r(x):=\text{dist}(o,x)$ for $x\in \mathbb{S}^n$. In this section, we prove the following result.

\begin{proposition}\label{unitsphereexamp}
Given $p>n$, we have
\begin{align*}
\int_{\mathbb{S}^{n}}|\nabla u|^p\dvol_g\geq \left( \frac{p-n}{p} \right)^p\int_{\mathbb{S}^{n}}  \frac{|u|^p}{{r}^p}\dvol_g,\ \forall\,u\in C^\infty(\mathbb{S}^{n},o).\tag{4.1}\label{Hardynosphere}
\end{align*}
In particular, $\left( \frac{p-n}{p} \right)^p$ is sharp.
\end{proposition}

\begin{proof}
For convenience, we use $\Delta_p$ and $\di$ to denote the $p$-Laplacian and the divergence operators with respect to $\vol_g$, that is,
\[
\Delta_p:=\Delta_{\vol_g,p},\ \di:=\di_{\vol_g}.
\]

We first consider the case when $p$ is even.
Set $\rho:=r^{\frac{p-n}{p-1}}$.
Thus,
\[
-\Delta_p \rho=(n-1)\left(\frac{p-n}{p-1}\right)^{p-1}r^{-n}(1-r\cot r)\geq 0, \text{ on }\mathbb{S}^n_o:=\mathbb{S}^n\backslash\{o\}.\tag{4.2}\label{plalacianonS}
\]
Given $u\in C^\infty(\mathbb{S}^n,o)$, set $v:=u/\rho^\gamma$, where $\gamma:=\frac{p-1}{p}$.
 Since $p$ is even, one has
\begin{align*}
v^p=|v|^p=\frac{u^p}{r^{p-n}}.
\end{align*}
The Taylor expansion of $u$ about $o$ furnishes $\lim_{x\rightarrow o}v^p(x)=0$ and for a small $\epsilon>0$,
\[
|\nabla v^p|\leq {C_1}r^{n-1}, \text{ for }0<r<\epsilon,
\]
where $C_1=C_1(p,n,\nabla u(o))$ is a positive constant only dependent on $p$, $n$ and $\nabla u(o)$.
Hence,
\[
\left||\nabla\rho|^{p-2}g(\nabla\rho,\nabla v^p)\right|\leq |\nabla\rho|^{p-1} |\nabla v^p|\leq \left( \frac{p-n}{p-1} \right)^{p-1} C_1=:C_2,  \text{ for }0<r<\epsilon,
\]
which means $|\nabla\rho|^{p-2}g(\nabla\rho,\nabla v^p)\in L^1(\mathbb{S}^n)$. Similarly, one can show $\frac{|u|^p}{{\rho}^p}|\nabla{\rho}|^p\in L^1(\mathbb{S}^n)$ by the Taylor expansion.
Moreover, a direct calculation yields
\begin{align*}
&\lim_{\eta\rightarrow0^+}\left|\int_{B_\eta(o)}|\nabla\rho|^{p-2}g(\nabla\rho,\nabla v^p)\dvol_g\right|\leq \lim_{\eta\rightarrow0^+}C_2 \vol_g\left( B_\eta(o) \right)=0,\tag{4.3}\label{smallballto0}\\
&\lim_{\eta\rightarrow0^+}\left|\int_{S_\eta(o)}v^p g(|\nabla\rho|^{p-2}\nabla\rho,\nabla r)dA\right|\leq c_{n-1}\left( \frac{p-n}{p-1} \right)^{p-1}\lim_{\eta\rightarrow0^+}\left[\left( \frac{\sin \eta}{\eta}\right)^{n-1}\max_{S_\eta(o)}\left|v^p\right|\right]=0,\tag{4.4}\label{divegencelemmaonS}
\end{align*}
where $S_\eta(o):=\partial B_\eta(o)$ and $dA$ is the area measure on $S_\eta(o)$ induced by $\dvol_g$.
On one hand, (\ref{plalacianonS}) implies
\[
|\nabla\rho|^{p-2}g(\nabla\rho,\nabla v^p)=\di\left(v^p|\nabla\rho|^{p-2}\nabla\rho\right)-v^p\Delta_p \rho\geq \di\left(v^p|\nabla\rho|^{p-2}\nabla\rho\right), \text{ on }\mathbb{S}^n_o,
\]
which together with (\ref{smallballto0}) and (\ref{divegencelemmaonS})  yields
\begin{align*}
\infty&>\int_{\mathbb{S}^n}|\nabla\rho|^{p-2}g(\nabla\rho,\nabla v^p)\dvol_g= \lim_{\eta\rightarrow0^+}\left(\int_{\mathbb{S}^n\backslash B_\eta(o)} +\int_{B_\eta(o)}\right) \\
\geq &\lim_{\eta\rightarrow0^+}\left(\int_{\mathbb{S}^n\backslash B_\eta(o)} \di\left(v^p|\nabla\rho|^{p-2}\nabla\rho\right) \dvol_g+\int_{B_\eta(o)}|\nabla\rho|^{p-2}g(\nabla\rho,\nabla v^p)\dvol_g\right)\\
=&\lim_{\eta\rightarrow0^+}\int_{S_\eta(o)}v^p g(|\nabla\rho|^{p-2}\nabla\rho,\nabla r)dA+\lim_{\eta\rightarrow0^+}\int_{B_\eta(o)}|\nabla\rho|^{p-2}g(\nabla\rho,\nabla v^p)\dvol_g=0.\tag{4.5}\label{nonnegiatve}
\end{align*}
On the other hand, the same argument as in Lemma  \ref{leBV} yields
\begin{align*}
|\nabla u|^p\geq \gamma^p \frac{|u|^p}{{\rho}^p}|\nabla{\rho}|^p+\gamma^{p-1}  |\nabla{\rho}|^{p-2}g(\nabla{\rho}, \nabla v^p), \text{ on }\mathbb{S}^{n}_o.\tag{4.6}\label{inequalitysee}
\end{align*}
Since every term on both sides of (\ref{inequalitysee}) is in $L^1(\mathbb{S}^n)$ and $o$ is of zero measure, by integrating (\ref{inequalitysee}) over $\mathbb{S}^n$ and using (\ref{nonnegiatve}), we have
\begin{align*}
\int_{\mathbb{S}^{n}}|\nabla u|^p\dvol_g\geq \left( \frac{p-1}{p} \right)^p\int_{\mathbb{S}^{n}}  \frac{|u|^p}{{\rho}^p}|\nabla{\rho}|^p\dvol_g. \tag{4.7}\label{finhardyonsphere}
\end{align*}

By a similar but more tedious argument, one can check that
all the estimates above remain valid even if $p$ is not even. Hence, (\ref{finhardyonsphere}) holds for all $p>n$.
By repeating the proof of Lemma \ref{secondlemmaforsharp}, one can show that sharpness of (\ref{finhardyonsphere}).
\end{proof}

\appendix
\section{}\label{app}

\begin{lemma}\label{lpschcom}
Let $(M,o,g,d\mu)$, $p,\alpha,\beta$ and $\rho$  be as in Definition \ref{DefDS}. Then for any globally Lipschitz function $u$ on $M$  with compact support in $M_o$, we have $u\in D^{1,p}(M_o,\rho^{p+\beta})$.
\end{lemma}
\begin{proof}
Since $\text{supp}(u)$ is compact, there exist  a coordinate covering   $\{(U_k,\phi_k)\}_{k=1}^{N<\infty}$ of $\text{supp}(u)$   such that    $U_k\subset\subset M_o$ and $\phi_k(U_k)=\mathbb{B}_{\mathbf{0}}(1)\subset \mathbb{R}^n$  for each $k\in \{1,\ldots,N\}$. We can also assume that $K:=\cup_{k}\overline{U_k}$ is a compact subset in $M_o$.

Let $\{\eta_k\}_{k=1}^N$ be a smooth partition of unity subordinate to $\{U_k\}_{k=1}^N$. Thus,  $(\eta_ku)\circ\phi_k^{-1}$ is a globally Lipschitz function on $\mathbb{B}_{\mathbf{0}}(1)$ with respect to the Euclidean distance and hence, $(\eta_ku)\circ\phi_k^{-1}$ belongs to the  Sobolev space $ W^{1,pq'}(\mathbb{B}_{\mathbf{0}}(1))$, where $q'=q/(q-1)$.
Meyers-Serrin's theorem then yields a sequence $v_{k_j}\in C_0^\infty(\mathbb{B}_{\mathbf{0}}(1))$ with $\lim_{j\rightarrow\infty}\|v_{k_j}-(\eta_ku)\circ\phi_k^{-1}\|_{W^{1,pq'}(\mathbb{B}_{\mathbf{0}}(1))}=0$. Therefore, we have $v_{k_j}\circ\phi_k\in C^\infty_0(M_o)$ with $\text{supp}(v_{k_j}\circ\phi_k)\subset U_k$, which  implies
\begin{align*}
&\|v_{k_j}\circ\phi_k-(\eta_ku)\|^p_D=\int_K |\nabla(v_{k_j}\circ\phi_k)-\nabla(\eta_ku)|^{p}\rho^{p+\beta} d\mu\\
\leq& \left(\int_K |\nabla(v_{k_j}\circ\phi_k)-\nabla(\eta_ku)|^{pq'} d\mu \right)^{\frac1{q'}}\left( \int_K \rho^{q(p+\beta)} d\mu\right)^{\frac1q} \\
\leq& C\,\|v_{k_j}-(\eta_ku)\circ\phi_k^{-1}\|^p_{W^{1,pq'}(\mathbb{B}_{\mathbf{0}}(1))}\left( \int_K \rho^{q(p+\beta)} d\mu\right)^{\frac1q}\rightarrow 0, \text{ as }j\rightarrow\infty,
\end{align*}
where $C=C\left(g|_{K},d\mu|_{K}\right)$ is a positive constant only dependent on $g|_K$ and $d\mu|_{K}$.
Hence, $(\eta_ku)\in D^{1,p}(M_o,\rho^{p+\beta})$. We conclude the proof by $u=\sum_{k=1}^N(\eta_ku)$.
\end{proof}

\begin{lemma}\label{controlnormandspace}
Let $(M,o,g,d\mu)$, $p,\alpha,\beta$ and $\rho$  be as in Definition \ref{DefDS}. Denote by $L^p(M_o,\rho^\beta|\nabla\rho|^p)$  the closure of  $C^\infty_0(M_o)$ with respect to the norm
\[
\|u\|_L:=\left( \int_{M_o} | u|^p \rho^{\beta}|\nabla\rho|^p d\mu\right)^{\frac1p}.
\]
Thus, $D^{1,p}(M_o,\rho^{p+\beta})\subset L^p(M_o,\rho^\beta|\nabla\rho|^p)$ and moreover, $\|f\|_{L}\leq \vartheta_{\alpha,\beta,p}\|f\|_D$ for any $f\in D^{1,p}(M_o,\rho^{p+\beta})$.
\end{lemma}
\begin{proof} The lemma is a special case of Berezansky, Sheftel and Us\cite[Theorem 7.1]{BSU} due to Lemma \ref{mainlemmforcr}. According to Berezansky et al. \cite[Theorem 7.2]{BSU}, it suffices to show that if a sequence $f_n\in C^\infty_0(M_o)$ is  fundamental respect to $\|\cdot\|_D$ and approaches to $0$ with respect to $\|\cdot\|_L$, then $\{f_n\}_n$ also approaches to $0$ with respect to $\|\cdot\|_D$.

In order to prove this, let $\|\cdot\|_{1,p,\beta}$ be a norm  on $C^\infty_0(M_o)$ defined by
\[
\|u\|_{1,p,\beta}:=\|u\|_{p,\beta}+\|\nabla u\|_D:=\left( \int_{M_o} |u|^p\rho^{p+\beta} d\mu \right)^{\frac1p}+\left( \int_{M_o} |\nabla u|^p\rho^{p+\beta} d\mu \right)^{\frac1p}.
\]
 Since $\rho/|\nabla\rho|\leq C$, we have
\[
\|f\|_D\leq \|f\|_{1,p,\beta}\leq (1+C\vartheta_{\alpha,\beta,p})\|f\|_D,\ \forall\,f\in C^\infty_0(M_o).\tag{A.1}\label{equvilence}
\]
Namely, $\|\cdot\|_D$ is equivalent to $\|\cdot\|_{1,p,\beta}$. Suppose that $\{f_n\}_n$ converges to $h\in D^{1,p}(M_o,\rho^{p+\beta})$ under $\|\cdot\|_D$.
Thus, the sequence $\{f_n\}_n$  also converges to $h$ with respect to $\|\cdot\|_{1,p,\beta}$ and hence, $\|f_n-h\|_{p,\beta}\rightarrow 0$.
Now $h=0$ follows from $\|f_n\|_{p,\beta}\leq C\|f_n\|_L\rightarrow 0$.
\end{proof}
\begin{remark}\label{newimoprremar}
 In the above proof, we do not require $\rho^{(p+\beta)/(1-p)},\rho^{q(p+\beta)}\in L^1(M)$ for some $q>1$.
On the other hand, if $\rho$ additionally satisfies this assumption,  then (\ref{equvilence}) implies $W^{1,p}(M_o,\rho^{p+\beta})=D^{1,p}(M_o,\rho^{p+\beta})$, where $W^{1,p}(M_o,\rho^{p+\beta})$ is defined as in Definition \ref{newst2defi}.
\end{remark}


In the sequel, let $(M,o,g,d\mu),p,\beta$ and $\rho$ be as in Definition \ref{newst2defi}.
We now define the weighted $L^p$-space $L^p(M,\rho^{p+\beta})$ (resp., $L^p(TM,\rho^{p+\beta})$) as the completion of $C^\infty(M)$ (resp., $\Gamma^\infty(TM)$, i.e., the space of the smooth sections of the tangent bundle) under the norm
\[
\|u\|_{p,\beta}:=\left( \int_{{M}}|u|^p{\rho}^{p+\beta} d\mathfrak{m}\right)^{\frac1p} \ \ \left(\text{resp., } \|X\|_{p,\beta}:=\||X|\|_{p,\beta}=\left( \int_{{M}}|X|^p{\rho}^{p+\beta} d\mathfrak{m}\right)^{\frac1p} \right).
\]
And set $L^p(M):=L^p(M,\rho^0)$ and $L^p(TM):=L^p(TM,\rho^0)$.

\begin{lemma}\label{Soleweakder}Let $(M,o,g,d\mu),p,\beta$ and $\rho$ be as in Definition \ref{newst2defi}.
If $u\in W^{1,p}({M},{\rho}^{p+\beta})$, then $u\in W^{1,1}(M)$. Moreover, the gradient $\varpi$ of $u$ in $W^{1,p}({M},{\rho}^{p+\beta})$ is the distributional derivative of $u$, i.e.,  $\varpi\in L^1(TM)$ and
\[
\int_{{M}} g( X,\varpi) d{\mu}=-\int_{{M}}u \di_{\mu} X d{\mu},\ \text{ for any smooth vector field }X.
\]
\end{lemma}
\begin{proof}
Given $f\in L^p({M},\rho^{p+\beta})$, the H\"older inequality yields
\begin{align*}
\int_M|f| d{\mu}&=\int_M |f|{\rho}^{\frac{\beta+p}{p}}{\rho}^{-\frac{\beta+p}{p}}d{\mu}\leq \left( \int_M |f|^p {\rho}^{p+\beta}d{\mu} \right)^{\frac1p}\left( \int_M \rho^{\frac{p+\beta}{1-p}} d{\mu} \right)^{\frac{p-1}p}. \tag{A.2}\label{newappedix}
\end{align*}
Consequently, if $u\in W^{1,p}({M},{\rho}^{p+\beta})$,   (\ref{newappedix}) implies $u,|\varpi|\in L^{1}({M})$.
On the other hand,
there exists a sequence $u_j\in C^\infty_0({M})$ such that $\|u_j-u\|_{p,\beta}+\|\nabla u_j-\varpi\|_{p,\beta}=\|u_j-u\|_{1,p,\beta}\rightarrow0$. Thus, for any smooth vector field $X$,  (\ref{newappedix}) together with the compactness of $M$ yields
\begin{align*}
&\left| \int_{{M}} g(  X,\varpi)-(-u{\di}_{\mu} X) d{\mu} \right|=\left| \int_{{M}}  g( X,\varpi)- g( X,\nabla u_j)+ g( X,\nabla u_j) -(-u{\di}_{\mu} X) d{\mu}  \right|\\
\leq &\int_M \left|g(  X, \nabla u_j-\varpi) \right|d{\mu}+ \int_M \left|(u_j-u){\di}_{\mu} X \right|d{\mu}\\
\leq &\max_M |X| \int_{M}|\nabla u_j-\varpi|d{\mu} +\max_M|{\di}_{\mu} X| \int_{M}|u_j- u|d{\mu} \\
\leq &\left(\max_M |X|+\max_M|{\di}_{\mu} X|\right)\left( \int_M {\rho}^{\frac{p+\beta}{1-p}} d{\mu} \right)^{\frac{p-1}p}\left(\|u_j -u\|_{p,\beta} +\|\nabla u_j-\varpi \|_{p,\beta}\right)\rightarrow 0.
\end{align*}
Furthermore, (\ref{newappedix}) also implies that $u_j\rightarrow u$ in $W^{1,1}({M})$ and hence, the lemma follows.
\end{proof}


\begin{lemma}\label{maxminle}Let $(M,o,g,d\mu),p,\beta$ and $\rho$ be as in Definition \ref{newst2defi}.
If $u\in W^{1,p}({M},\rho^{\beta+p})$, then $|u|$, $u_+:=\max\{u,0\}$ and $u_-:=-\min\{u,0\}$  are all in $W^{1,p}({M},\rho^{\beta+p})$.
\end{lemma}
\begin{proof} Note that $u_+=\frac12(u+|u|)$ and $u_-=\frac12(|u|-u)$.  Hence, it suffices to show $|u|\in W^{1,p}({M},\rho^{\beta+p})$.
Set $q':=q/(q-1)$. Thus, for any   $f\in L^{pq}(M)$, the H\"older inequality yields
\begin{align*}
\int_M |f|^p{\rho}^{\beta+p}d{\mu}\leq \left( \int_M  |f|^{pq'} d{\mu}\right)^{\frac1{q'}} \left(  \int_M {\rho}^{ {q(\beta+p)}}d{\mu}\right)^{\frac{1}{q}}.\tag{A.3}\label{newholderB2}
\end{align*}

 First we consider the case when $u\in C^\infty_0({M})$.
The standard theory yields a sequence $u_j\in C^\infty_0({M})$ such that
$u_j\rightarrow |u|$ in $W^{1,pq'}({M})$ (cf. Hebey \cite[Lemma 2.5]{H}),
which together with (\ref{newholderB2}) implies  $u_j\rightarrow |u|$ in $W^{1,p}({M}, {\rho}^{p+\beta})$. Hence, $|u|\in W^{1,p}({M}, {\rho}^{p+\beta})$.

 For the general case (i.e., $u\in W^{1,p}({M},{\rho}^{\beta+p})$), choose a sequence $u_j\in C^\infty_0({M})$ such that $\|u_j- u\|_{1,p,\beta}\rightarrow 0$.  Thus, Lemma \ref{Soleweakder} and (\ref{newappedix}) imply $u_j\rightarrow u$
 in $W^{1,1}(M)$.
By  Lieb et al.
\cite[Theorem 6.17]{LL}, we have $\nabla|u_j|=\text{sgn}(u_j)\nabla u_j$ and $\nabla|u|=\text{sgn}(u)  \nabla u$. From this, by passing a subsequence,
one can show that      $|u_{j}|\rightarrow |u|$ and $\nabla|u_{j}|\rightarrow \nabla|u|$ pointwise a.e., which
implies $(\text{sgn}(u_j)-\text{sgn}(u))\nabla u\rightarrow 0$ pointwise a.e. Since
\begin{align*}
\left|(\text{sgn}(u_j)- \text{sgn}(u))\nabla u \right|^p\rho^{p+\beta} \leq 2^p|\nabla u |^p \rho^{p+\beta}\in L^1(M),
 \end{align*}
 the dominated convergence theorem  yields
\begin{align*}
 &\lim_{j\rightarrow\infty}\int_{M}| \nabla|{u_{j}}|-\nabla |u||^p\rho^{p+\beta}d\mu=\lim_{j\rightarrow\infty}\int_{M}|  \text{sgn}(u_j)\nabla {u_{j}} -\text{sgn}(u)\nabla u|^p\rho^{p+\beta}d\mu\\
 \leq &\lim_{j\rightarrow\infty}2^p\left[\int_{M} |(\text{sgn}(u_j)- \text{sgn}(u))\nabla u |^p \rho^{p+\beta}d\mu  +\int_{M} |\nabla u_j-\nabla u|^p    \rho^{p+\beta}d\mu       \right]=0,
 \end{align*}
which together with $\|u_j- u\|_{p,\beta}\rightarrow 0$ implies $\||u_j|- |u|\|_{1,p,\beta}\rightarrow 0$. Since $|u_j|\in W^{1,p}({M}, {\rho}^{p+\beta})$,
we are done.
\end{proof}

\noindent{\textbf{Acknowledgements}}
 This work was supported by  NNSFC (No. 11761058), NSFS (No. 19ZR1411700) and National College Students' innovation and entrepreneurship training program.

\end{document}